\documentclass[reqno]{amsart}%
\usepackage{amssymb}
\usepackage{amsfonts}
\usepackage{amsmath}
\usepackage{graphicx}%
\providecommand{\U}[1]{\protect\rule{.1in}{.1in}}
\theoremstyle{plain}

\newtheorem{claim}{Claim}

\newtheorem{corollary}{Corollary}

\newtheorem{definition}{Definition}
\newtheorem{example}{Example}

\newtheorem{lemma}{Lemma}

\newtheorem{proposition}{Proposition}
\newtheorem{remark}{Remark}

\newtheorem{theorem}{Theorem}
\numberwithin{equation}{section}
\begin{document}
\title[Bilipschitz Embedding of Homogeneous Fractals]{Bilipschitz Embedding of Homogeneous Fractals}
\author{Fan L\"u}
\address{School of Mathematics and Statistics, Huazhong University of Science and
Technology, 430074, Wuhan, P. R. China}
\email{lvfan1123@163.com}
\author{Man-Li Lou}
\address{Department of Mathematics, Guangdong Polytechnic Normal University, Guangzhou,
510665, P. R. China}
\email{loumanli@126.com}
\author{Zhi-Ying Wen}
\address{Department of Mathematics, Tsinghua University, 100084, Beijing, P. R. China}
\email{wenzy@tsinghua.edu.cn}
\author{Li-Feng Xi}
\address{Institute of Mathematics, Zhejiang Wanli University, 315100, Ningbo, P. R. China}
\email{xilifengningbo@yahoo.com}
\thanks{Man-Li Lou is corresponding author. The work is supported by NSFC (Nos.
11371329, 11271223, 11071224, 11101159), NCET, NSF of Zhejiang Province (Nos.
LR13A1010001, LY12F02011) and NSF of Guangdong Province (No. S2011040005741).}

\begin{abstract}
In this paper, we introduce a class of fractals named homogeneous sets based
on some measure versions of homogeneity, uniform perfectness and doubling.
This fractal class includes all Ahlfors--David regular sets, but most of them
are irregular in the sense that they may have different Hausdorff dimensions
and packing dimensions. Using Moran sets as main tool, we study the
dimensions, bilipschitz embedding and quasi-Lipschitz equivalence of
homogeneous fractals.

\end{abstract}
\maketitle

\section{Introduction}

It is well known that self-similar sets and self-conformal sets satisfying the
open set condition (\textbf{OSC}) are always \emph{Ahlfors--David}
\emph{regular }\cite{Hutchinson, Mattila}. We say that a compact subset $A$ of
metric space $(X,$d$)$ is Ahlfors--David $s$-regular with $s\in(0,\infty),$ if
there is a Borel measure $\mu$ supported on $A$ and a constant $c\geq1$ such
that for all $x\in A$ and $0<r\leq|A|,$
\begin{equation}
c^{-1}r^{s}\leq\mu(B(x,r))\leq cr^{s},\text{ } \label{A-D}%
\end{equation}
where $B(x,r)$ is the closed ball centered at $x$ with radius $r$ and
$|\cdot|$ denotes the diameter of a set. For an Ahlfors--David $s$-regular set
$A,$ $0<\mathcal{H}^{s}(A)<\infty$ and $\dim_{H}A=\dim_{P}A=s$, i.e., its
Hausdorff dimension and packing dimension are the same.

Ahlfors--David regularity is a weak notion of \emph{homogeneity} \cite{David
Semmes}. We give another \emph{measure version} of homogeneity, i.e., there is
a constant $\lambda\geq1$ such that for all $x_{1},x_{2}\in A$ and
$0<r\leq|A|,$
\begin{equation}
\lambda^{-1}\leq\frac{\mu(B(x_{1},r))}{\mu(B(x_{2},r))}\leq\lambda.
\label{self}%
\end{equation}
Naturally, (\ref{self}) holds for all Ahlfors--David regular sets.

We also need two other notions, \emph{uniform perfectness} and \emph{doubling}%
, which play important roles in the research of metric spaces. For example,
Proposition $15.11$ of \cite{David Semmes} shows that if a compact metric
space is uniformly perfect, doubling and uniformly disconnected, then it is
quasisymmetrically equivalent to a symbolic system $\Sigma_{2}$.

We notice that any Ahlfors--David regular set $A$ is \emph{uniformly perfect}
(see, e.g., \cite{BP} and \cite{MR}), i.e., it contains more than one point
and there exists a constant $t\in(0,1)$ such that $[B(x,r)\backslash
B(x,tr)]\cap A\neq\varnothing$ for all $x\in A,$ $0<r\leq|A|.$ For the
\emph{measure version} of uniform perfectness, we obtain an alternative
condition: there exists a constant $\kappa_{1}<1$ such that
\begin{equation}
\inf_{x\in A,\;r\leq|A|}\frac{\mu(B(x,r))}{\mu(B(x,\kappa_{1}r))}>1. \label{*}%
\end{equation}
It follows from (\ref{A-D}) that any Ahlfors--David regular set satisfies
(\ref{*}).

In a metric space, the notion of \emph{doubling} describes that any closed
ball of radius $r$ can be covered by no more than $M$ balls of radius $r/2,$
where $M$ is a constant. The notion of doubling also has \emph{measure
version, }see e.g. \cite{LS} and \cite{Wu}\emph{. }For compact subsets in
metric space, these two versions are equivalent. It follows from (\ref{A-D})
that any Ahlfors--David regular measure is doubling, i.e., there exists a
constant $T\geq1$ such that $\mu(B(x,r))\leq T\mu(B(x,r/2))$ for all $x\in A,$
$0<r\leq|A|,$ i.e., for $\kappa_{2}=1/2,$
\begin{equation}
\sup_{x\in A,\;r\leq|A|}\frac{\mu(B(x,r))}{\mu(B(x,\kappa_{2}r))}<\infty.
\label{**}%
\end{equation}

Simulating the homogeneity, uniform perfectness and doubling by (\ref{self}),
(\ref{*}) and (\ref{**}), we can define a large class of fractals, which are
not so good as Ahlfors--David regular sets but \emph{homogeneous} in certain sense.

\begin{definition}
\label{D:Main}A compact subset $A$ of metric space $(X,$\emph{d}$)$ is said to
be \textbf{homogeneous}\emph{, }if $|A|>0$ and there is a Borel probability
measure $\mu$ supported on $A$ satisfying:

\begin{enumerate}
\item There is a constant $\lambda_{A}\geq1,$ such that for all $x_{1}%
,x_{2}\in A$ and $0<r\leq|A|,$
\begin{equation}
\lambda_{A}^{-1}\leq\frac{\mu(B(x_{1},r))}{\mu(B(x_{2},r))}\leq\lambda_{A};
\label{cc}%
\end{equation}

\item There are constants$\ \kappa_{A}\in(0,1)$ and $1<\delta_{A}\leq
\Delta_{A}<\infty,$ such that for all $x\in A$ and $0<r\leq|A|,$
\begin{equation}
\delta_{A}\leq\frac{\mu(B(x,r))}{\mu(B(x,\kappa_{A}r))}\leq\Delta_{A}.\text{ }
\label{***}%
\end{equation}

\end{enumerate}
\end{definition}

\begin{figure}[ptbh]
\centering\includegraphics[width=90mm]{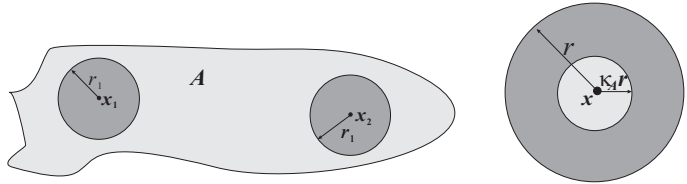}\caption{Compare the measures of
different balls}
\end{figure}

\begin{remark}
All Ahlfors--David regular sets are homogeneous. But homogeneous sets may be
not Ahlfors--David regular, see Proposition $\ref{P;homo not regular}$ and
Example $\ref{EX:k+1/k+2 copy(1)}$ in Section $3.$
\end{remark}

\begin{remark}
Any Moran set is homogeneous, see Proposition $\ref{P:Mor is Homo}$ in
Subsection $1.2.$
\end{remark}

\begin{remark}
If there exists a point $x$\ in $A$ such that $\delta\leq\mu(B(x,r))/\mu
(B(x,\kappa r))\leq\Delta$ holds for all $0<r\leq|A|$ with constants
$\kappa\in(0,1)$ and $1<\delta\leq\Delta<\infty,$ then it follows from
$(\ref{cc})\ $that $(\ref{***})$ holds for any point in $A$ $($with some
constants $\kappa_{A},\delta_{A}$ and $\Delta_{A}).$
\end{remark}

\begin{remark}
By $(\ref{***})$, there are no atoms in $A$.
\end{remark}

There are some fundamental questions about homogeneous fractals:

\begin{itemize}
\item How about the \textbf{dimensions} of homogeneous fractals? Can we find
\textbf{a large class of homogeneous fractals} which are not Ahlfors--David regular?

\item How about the \textbf{bilipschitz embedding} between homogeneous
fractals? Which kind of good fractals can be bilipschitz embedded into them?

\item Given two homogeneous fractals, when are they bilipschitz equivalent? An
alternative but weaker question is of \textbf{quasi-Lipschitz equivalence}.
\end{itemize}

To answer the above questions, we define a function $\alpha_{A}(x,r)$ for a
homogeneous set $A$ as follows:
\begin{equation}
\alpha_{A}(x,r)=\log\mu(B(x,r))/\log r\qquad\text{ for } x\in A,\ 0<r\leq|A|.
\label{E:equ0}%
\end{equation}
Here $\alpha_{A}(x,r)$ is similar to the function with respect to pointwise dimension.

For any function $g(r)$ defined on $(0,\delta)$ with $\delta>0,$ we focus on
the behavior of the function $g(r)$ when $r\rightarrow0.$ In fact, for any
function $h(r)$ with
\begin{equation}
\left\vert h(r)-g(r)\right\vert =O(|\log r|^{-1}), \label{E:equ}%
\end{equation}
we denote $g\sim h$ and define an equivalence class $[g]=\{h:g\sim h\}.$ Then,
as $\alpha_{A}(x_{1},r)\sim\alpha_{A}(x_{2},r)$ by (\ref{cc}), we use
$\alpha_{A}(r)$ to denote any one function in the equivalence class
$[\alpha_{A}(x,r)]$ with $x\in A.$ For example, we can take
\begin{equation}
\alpha_{A}(r)\equiv s\text{ for an Ahlfors--David }s\text{-regular set }A.
\label{regular}%
\end{equation}

With the help of the function $\alpha_{A}(r)$ defined above, we can answer the
above questions on dimensions, bilipschitz equivalence and quasi-Lipschitz equivalence.

\subsection{Dimensions}

\begin{proposition}
\label{P:Prop of Homo}For a homogeneous set $A$, we have:

\begin{enumerate}
\item $0<\liminf_{r\rightarrow0}\alpha_{A}(r)\leq\limsup_{r\rightarrow0}%
\alpha_{A}(r)<\infty$ and
\[
\qquad\dim_{H}A=\underline{\dim}_{B}A=\liminf_{r\rightarrow0}\alpha
_{A}(r),\text{ }\dim_{P}A=\overline{\dim}_{B}A=\limsup_{r\rightarrow0}%
\alpha_{A}(r),
\]
where $\dim_{P}A$ denotes the radius packing dimension of metric space $A$
defined in \cite{Cutler}, which coincides with the usual definition when $A$
is a subset of a Euclidean space.

\item Suppose $N(A,r)$ is the smallest number of balls with radius $r$ needed
to cover $A.$ Let $f_{A}(r)=\frac{\log N(A,r)}{-\log r}.$ Then
\begin{equation}
f_{A}(r)\sim\alpha_{A}(x,r)\text{ for any }x\in A. \label{behavior}%
\end{equation}

\end{enumerate}
\end{proposition}

These properties show that for a homogeneous set $A$,

\begin{itemize}
\item The behavior of $\alpha_{A}(r)$ when $r\rightarrow0$ is only determined
by $N(A,r)$ as in (\ref{behavior}), i.e., $\alpha_{A}(x,r)\sim\frac{\log
N(A,r)}{-\log r},$ depending on the geometric structure of $A$ and not
depending on the choice of the Borel measure $\mu;$

\item The behavior of $\alpha_{A}(r)$ when $r\rightarrow0$ plays a role more
important than fractal dimensions. We concern not only the dimension values
$\liminf_{r\rightarrow0}\alpha_{A}(r)$ and $\limsup_{r\rightarrow0}\alpha
_{A}(r),$ but also the behavior of $\alpha_{A}(r)$ when $r\rightarrow0.$
\end{itemize}

\subsection{\textbf{Moran sets are homogeneous}}

\

Moran sets were first studied in \cite{Moran} by Moran. We recall this fractal class.

Fix a compact set $J\subset\mathbb{R}^{d}$ with its interior non-empty. Fix a
ratio sequence $\{c_{k}\}_{k\geq1\text{ }}$and an integer sequence
$\{n_{k}\}_{k\geq1}$ satisfying $c_{k}\in(0,1)$ and $n_{k}\geq2$ for all $k.$
For $D_{1},$ $D_{2}\subset\mathbb{R}^{d}$, we say that $D_{1}$ is
geometrically similar to $D_{2}$ with ratio $r$, if there is a similitude
$S$\ with ratio $r$ such that $D_{1}=S(D_{2}).$ Let $\Omega_{0}=\{\emptyset\}$
with the empty word $\emptyset$, and let $\Omega_{k}=\{$word $i_{1}\cdots
i_{k}:$ for the $t$-th letter, $i_{t}\in\mathbb{N}\cap\lbrack1,n_{t}]$ for all
$t\}$ for $k\geq1$. In this paper, we always assume that
\begin{equation}
c_{\ast}=\inf_{k}c_{k}>0. \label{c_*}%
\end{equation}

Suppose there are $J_{1},J_{2},\cdots,J_{n_{1}}\subset J_{\emptyset}=J$
geometrically similar to $J$ with ratio $c_{1}$ and their interiors being
pairwise disjoint. Inductively, for any $i_{1}\cdots i_{k-1}\in\Omega_{k-1}$,
suppose there are $J_{i_{1}\cdots i_{k-1}1},J_{i_{1}\cdots i_{k-1}2}%
,,\cdots,J_{i_{1}\cdots i_{k-1}n_{k}}\subset J_{i_{1}\cdots i_{k-1}}$
geometrically similar to $J_{i_{1}\cdots i_{k-1}}$ with ratio $c_{k}$ and
their interiors being pairwise disjoint. Then
\begin{equation}
E= {\displaystyle\bigcap\nolimits_{k=0}^{\infty}} {\displaystyle\bigcup
\nolimits_{i_{1}\cdots i_{k}\in\Omega_{k}}} J_{i_{1}\cdots i_{k}} \label{E}%
\end{equation}
is called a Moran set. We denote $E\in\mathcal{M}(J,\{n_{k}\}_{k}%
,\{c_{k}\}_{k}).$

\begin{figure}[ptbh]
\centering\includegraphics[width=80mm]{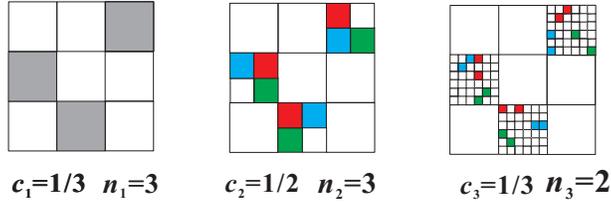}\caption{The first three
steps of the construction of a Moran set with $J=[0,1]^{2}$}%
\end{figure}

Many classical self-similar sets are Moran sets. For the Cantor ternary set
and the von Koch curve, setting $n_{k}\equiv2$ in both cases, letting
$c_{k}\equiv1/3$ or $c_{k}\equiv1/\sqrt{3}$, respectively, and taking $J$ as
$[0,1]$ or a suitable solid triangle, respectively, we get their structures.
For details of a more general structure, please refer to \cite{Wen}.

Under the assumption (\ref{c_*}), we have

\begin{proposition}
\label{P:Mor is Homo}Any Moran set is homogeneous. Suppose $E$ is a Moran set
defined above. Then we can take $\alpha_{E}(r)=\frac{\log n_{1}\cdots n_{k}%
}{-\log c_{1}\cdots c_{k}}$ if $(c_{1}\cdots c_{k})|J|<r\leq(c_{1}\cdots
c_{k-1})|J|.$
\end{proposition}

Note that bilipschitz image of a homogeneous set is homogeneous (Lemma
\ref{L:Bili of homo}).

\begin{corollary}
Any bilipschitz image of a Moran set is homogeneous.
\end{corollary}

\subsection{Approximation theorem}

$\ $\

How to describe the distance between two homogeneous sets? As usual, we can
use Hausdorff distance d$_{H}.$ For homogeneous sets $A$ and $B$ in a metric
space $(X,$d$)$,
\[
\text{\textrm{d}}_{H}(A,B)=\max\{\sup\nolimits_{x\in A}\text{\textrm{d}%
}(x,B),\sup\nolimits_{y\in B}\text{\textrm{d}}(y,A)\},
\]
where \textrm{d}$(x,B)=\inf_{z\in B}$\textrm{d}$(x,z).$

We give a new pseudo-distance. Given homogeneous sets $A$ and $B,$ we
consider
\[
\chi(A,B):=\limsup_{r\rightarrow0}\left\vert \log\frac{\alpha_{A}(r)}%
{\alpha_{B}(r)}\right\vert .
\]
It is easy to check that $\chi$ is a pseudo-distance on the space of all
homogeneous sets, i.e., $\chi(A,B)\geq0$, $\chi(A,B)=\chi(B,A)$ and
$\chi(A,B)+\chi(B,C)\geq\chi(A,C).$ In fact, $\chi(A,B)=0$ if and only if
$\lim\limits_{r\rightarrow0}\frac{\alpha_{A}(r)}{\alpha_{B}(r)}=1,$ i.e.,
$\lim_{r\rightarrow0}(\alpha_{A}(x,r)-\alpha_{B}(y,r))=0$ for all $x\in A$,
$y\in B.$

\begin{proposition}
\label{p:mei}Given a homogeneous set $A,$ for any $\varepsilon>0,$ there
exists $\delta>0$ such that if $\chi(A,B)<\delta$ for a homogeneous set $B,$
then
\[
|\dim_{P}A-\dim_{P}B|,|\dim_{H}A-\dim_{H}B|<\varepsilon.
\]

\end{proposition}

We can approximate Ahlfors--David regular sets by self-similar sets as in
\cite{Mattila Saarenen}. For homogeneous sets, we replace self-similar sets by
Moran sets.

\begin{theorem}
\label{T:E Emb A}Suppose $A$ is a homogeneous set. Then for any $\varepsilon
>0,$ we can find a Moran set $E$ in a Euclidean space and a bilipschitz map
$f$ from $E$ to $A$ such that
\[
\mathrm{d}_{\text{H}}(f(E),A)<\varepsilon\quad\text{ and }\quad\chi
(E,A)=\chi(f(E),A)<\varepsilon.
\]
In particular, if $A$ is a homogeneous set in $\mathbb{R}^{d},$ for any
$\varepsilon>0,$ we can find\ a Moran set $F\subset A$ such that
\[
\mathrm{d}_{\text{H}}(F,A)<\varepsilon\text{\ and }0\leq(\dim_{P}A-\dim
_{P}F),(\dim_{H}A-\dim_{H}F)<\varepsilon.
\]

\end{theorem}

\subsection{Embedding theorem}

\ \

\begin{definition}
For metric spaces $(X,\mathrm{d}_{X})$ and $(Y,\mathrm{d}_{Y}),$ we say that
$X$ can be bilipschitz embedded into $Y,$\ denoted by $X\hookrightarrow Y,$ if
there exists an injection $f\colon X\rightarrow Y$ and a constant $L\geq1,$
such that for all $x_{1},x_{2}\in X,$%
\[
\mathrm{d}_{X}(x_{1},x_{2})/L\leq\mathrm{d}_{Y}(f(x_{1}),f(x_{2}))\leq
L\mathrm{d}_{X}(x_{1},x_{2}).
\]
Furthermore, if $f$ is a bijection, we say that $X$\ and $Y$ are bilipschitz equivalent.
\end{definition}

Mattila and Saaranen \cite{Mattila Saarenen}, Llorente and Mattila \cite{LM}
studied bilipschitz embeddings between subsets of Ahlfors--David regular sets
and self-conformal sets respectively. Inspired by \cite{Mattila Saarenen},
Deng, Wen, Xiong and Xi \cite{Deng juan} gave the results on self-similar sets.

In fact, Mattila and Saaranen \cite{Mattila Saarenen}\ obtained the following
interesting result: For an Ahlfors--David $s$-regular set $A$ and $t$-regular
set $B$ with $s<t$ and any $\varepsilon>0,$ there exists a self-similar set
$C_{\varepsilon}$ such that $\dim_{H}C_{\varepsilon}\in(s-\varepsilon,s]$ such
that $C_{\varepsilon}\hookrightarrow A$ and $C_{\varepsilon}\hookrightarrow
B;$ furthermore, if $A$ is \emph{uniformly disconnected}, then
$A\hookrightarrow B.$ An interesting fact is that $A$ is uniformly
disconnected if $s<1.$

It is natural to ask how to generalize the above bilipschitz embedding result
to homogeneous sets?

The following lemma is given a straightforward proof in Section 5.1.

\begin{lemma}
\label{L:blip}Let $A$ and $B$ be homogeneous with measures $\mu$ and $\nu,$
$\alpha_{A}(r)\sim\alpha_{A}(x^{\ast},r)$ and $\alpha_{B}(r)\sim\alpha
_{B}(y^{\ast},r)$ for some $x^{\ast}\in A$ and $y^{\ast}\in B$, and let
$A\hookrightarrow B.$ Then for any $x\in A,$ $y\in B$ and $r^{\prime}%
<r\leq\min(|A|,|B|),$
\begin{equation}
\frac{\mu(B(x,r))}{\mu(B(x,r^{\prime}))}\leq C\frac{\nu(B(y,r))}%
{\nu(B(y,r^{\prime}))} \label{E:unex}%
\end{equation}
where $C$ is an independent constant. Moreover, there is non-decreasing
function $\varepsilon:(0,\delta)\rightarrow(0,\infty)$ with $\delta\in(0,1)$
and $\varepsilon(r)\downarrow0$ as $r\downarrow0$ such that
\begin{equation}
\sup\limits_{r^{\prime}<r_{0}r<r<r_{0}}\left\vert \frac{\alpha_{A}(r)\log
r-\alpha_{A}(r^{\prime})\log r^{\prime}}{\alpha_{B}(r)\log r-\alpha
_{B}(r^{\prime})\log r^{\prime}}\right\vert \leq1+\varepsilon(r_{0})\text{ }%
\end{equation}
for every $r_{0}<\delta.$
\end{lemma}

Using the above lemma, we have

\begin{proposition}
\label{P:A<B}There is a homogeneous set $B$ and a number $t\in(0,\dim_{H} B)$
such that any Ahlfors--David regular set $A$, e.g., any self-similar set
satisfying the strong separation condition $($\textbf{SSC}$),$ can not be
bilipschitz embedded into $B,$ whenever $t<\dim_{H}A<\dim_{H}B$.
\end{proposition}

\begin{remark}
Compare this proposition with the result in \cite{Deng juan}$:$ Let $A$ and
$B$ be self-similar sets with $\dim_{H}A<\dim_{H}B$ and if $A$ satisfies
\textbf{SSC}, then $A\hookrightarrow B.$
\end{remark}

Now, we give the main result on the bilipschitz embedding.

\begin{theorem}
\label{T:A emb B}Suppose $A,B$ are homogeneous sets and $\alpha_{A}%
(r)\sim\alpha_{A}(x^{\ast},r)$ and $\alpha_{B}(r)\sim\alpha_{B}(y^{\ast},r)$
for some $x^{\ast}\in A$ and $y^{\ast}\in B.$ If%
\begin{equation}
\sup\limits_{r^{\prime}<r_{0}r<r<r_{0}}\left\vert \frac{\alpha_{A}(r)\log
r-\alpha_{A}(r^{\prime})\log r^{\prime}}{\alpha_{B}(r)\log r-\alpha
_{B}(r^{\prime})\log r^{\prime}}\right\vert <1 \label{rar}%
\end{equation}
for some $r_{0}<1,$ then for any $\varepsilon>0,$ there exists a homogeneous
subset $A_{\varepsilon}\subset A$ such that \emph{d}$_{H}(A_{\varepsilon
},A)<\varepsilon$, $\chi(A_{\varepsilon},A)<\varepsilon\ $and $A_{\varepsilon
}$ can be bilipschitz embedded into $B.$ Further, if $A$ is uniformly
disconnected and $(\ref{rar})$ holds, then $A\hookrightarrow B.$
\end{theorem}

\begin{remark}
\label{R:re}If $A,B$ are Ahlfors--David regular with $\dim_{H}A<\dim_{H}B,$
taking $\alpha_{A}(r)\equiv\dim_{H}A$ and $\alpha_{B}(r)\equiv\dim_{H}B,\ $we
obtain $(\ref{rar}).$
\end{remark}

\begin{remark}
Here $A_{\varepsilon}$ is bilipschitz equivalent to a Moran set. Compared with
\cite{Mattila Saarenen}, we use Moran sets to replace self-similar sets.
\end{remark}

Here we say that a compact subset $A$ of a metric space is \emph{uniformly
disconnected} \cite{David Semmes}, if there are constants $C>1$ and $r^{\ast
}>0$ so that for any $x\in A$ and $r<r^{\ast}$, there exists a set $E\subset
A$ satisfying
\begin{equation}
A\cap B(x,r)\subset E\subset B(x,Cr)\text{ and }\mathrm{d}(E,A\backslash E)>r.
\label{E:ud}%
\end{equation}
Any self-similar set satisfying \textbf{SSC} is uniformly disconnected.
Sometimes, we can use the uniform disconnectedness to replace \textbf{SSC}.

\begin{lemma}
\label{L:<1 Uni dis}Suppose $A$ is a homogeneous set. If
\[
\sup\limits_{r^{\prime}<r_{0}r<r<r_{0}}\left\vert \frac{\alpha_{A}(r)\log
r-\alpha_{A}(r^{\prime})\log r^{\prime}}{\log r-\log r^{\prime}}\right\vert
<1
\]
for some $r_{0}<1,$ then $A$ is uniformly disconnected.
\end{lemma}

\begin{remark}
\label{R:test}For Ahlfors--David $s$-regular set $A$\ with $s<1,$ we obtain
its uniform disconnectedness by taking $\alpha_{A}(r)\equiv s.$ This is a
result of \cite{Mattila Saarenen}. However, using a Moran set one can find a
homogeneous set $A$ with $\dim_{H}A=\dim_{P}A<1$ but which is not uniformly
disconnected $($see Example $\ref{Ex:ud}$ in Section $5).$
\end{remark}

For any $s\in(0,\infty),$ there exists a self-similar set $E$ with $\dim
_{H}E=s$ in Euclidean space satisfying \textbf{SSC}. Since $E$ is
Ahlfors--David regular and uniformly disconnected, applying Theorem
\ref{T:A emb B} and Lemma \ref{L:<1 Uni dis} to Ahlfors--David regular sets
(Remarks \ref{R:re} and \ref{R:test}), one can get a result of Mattila and
Saaranen \cite{Mattila Saarenen}.

\subsection{Equivalence theorem}

\ \

Classifying fractals under bilipschitz equivalence is an important topic in
geometric measure theory.

Bilipschitz mappings preserve the geometric properties, such as fractal
dimensions, Ahlfors--David regularity and uniform disconnectedness. Many works
have been devoted to the bilipschitz equivalence of fractals, please refer to
\cite{Falconer Marsh}, \cite{David Semmes}, \cite{Xi conformal}, \cite{RRX},
\cite{Xi Xiong} and \cite{LM}. But even for self-similar sets, Falconer and
Marsh \cite{Falconer Marsh} pointed out that there are two self-similar sets
satisfying \textbf{SSC} with the same Hausdorff dimension but which are not
bilipschitz equivalent.

Corresponding to bilipschitz equivalence, a weaker notion of quasi-Lipschitz
equivalence was introduced in \cite{Xi quasi}. Under quasi-Lipschitz mapping,
information of fractals is preserved in some sense, for example, the fractal
dimensions, quasi Ahlfors--David regularity, quasi uniform disconnectedness;
see e.g. \cite{Wang Xi S} and \cite{Xi quasi}.

\begin{definition}
\label{D:quasi lip}Two compact metric spaces $(X,\mathrm{d}_{X})$ and
$(Y,\mathrm{d}_{Y})$ are said to be quasi-Lipschitz equivalent, if there is a
bijection $f:X\rightarrow Y$ such that for all $x_{1},x_{2}\in X,$%
\[
\frac{\log\mathrm{d}_{Y}(f(x_{1}),f(x_{2}))}{\log\mathrm{d}_{X}(x_{1},x_{2}%
)}\rightarrow1\text{ uniformly as }\mathrm{d}_{X}(x_{1},x_{2})\rightarrow0.
\]

\end{definition}

If we turn to quasi-Lipschitz equivalence, we can say more about the
equivalence of homogeneous sets.

\begin{theorem}
\label{T:quasi-lip}Suppose homogeneous sets $A,B$ are uniformly
disconnected$.$ Then they are quasi-Lipschitz equivalent if and only if
$\chi(A,B)=0,$ i.e., $\lim_{r\rightarrow0}\frac{\alpha_{A}(r)}{\alpha_{B}%
(r)}=1.$
\end{theorem}

If $A$ and $B$ are Ahlfors--David $s$-regular and $t$-regular respectively, we
note that $\chi(A,B)=0$ if and only if $s=t.$ Using Theorem \ref{T:quasi-lip},
we get the main results of \cite{Wang Xi N}: Suppose that $A$ and $B$ are
Ahlfors--David $s$-regular and $t$-regular respectively, and that they are
uniformly disconnected; then they are quasi Lipschitz equivalent if and only
if they have the same Hausdorff dimension, i.e., $s=t$. In particular, the
assumption $s,t<1\ $impilies their uniform disconnectedness (see \cite{Mattila
Saarenen} or Remark \ref{R:test}). Then we also get the result of \cite{Xi
quasi}: Two self-conformal sets satisfying {\textbf{SSC}} are quasi-Lipschitz
equivalent if and only if they have the same Hausdorff dimension. For example,
the self-similar sets in Example 1 are quasi-Lipschitz equivalent.

\begin{example}
The Cantor ternary set and the self-similar set  $E=(rE)\cup(rE+\frac{1}%
{2}-\frac{r}{2})\cup(rE+1-r)$ with $r=3^{-\log3/\log2}$ are quasi-Lipschitz
equivalent, although they are not bilipschitz equivalent as shown in
\cite{Falconer Marsh} by Falconer and Marsh.
\end{example}

\subsection{Results on Moran sets}

$\ $

For a Moran class $\mathcal{A}=\mathcal{M}(J,\{n_{k}\}_{k\geq1},\{c_{k}
\}_{k\geq1}),$ supposing
\[
r\in(r_{k}|J|,r_{k-1}|J|],r^{\prime}\in(r_{k^{\prime}}|J|,r_{k^{\prime}%
-1}|J|]\text{ with }k\leq k^{\prime},
\]
where $r_{k}=c_{1}\cdots c_{k},$ we let
\[
\Phi_{\mathcal{A}}(r)=n_{1}\cdots n_{k}\text{ and }\Phi_{\mathcal{A}%
}(r,r^{\prime})=\Phi_{\mathcal{A}}(r^{\prime})/\Phi_{\mathcal{A}}%
(r)=n_{k+1}\cdots n_{k^{\prime}}.
\]

Applying Proposition \ref{P:Mor is Homo} to Theorems \ref{T:A emb B}%
--\ref{T:quasi-lip} and Lemma \ref{L:<1 Uni dis}, we have

\begin{corollary}
\label{C:Moran enbed}Let $\mathcal{A}=\mathcal{M}(J,\{n_{k}\}_{k}%
,\{c_{k}\}_{k})$ and $\mathcal{B}=\mathcal{M}(I,\{m_{k}\}_{k},\{d_{k}\}_{k}).$
If
\[
\limsup\limits_{k\rightarrow\infty}\frac{\log n_{k+1}\cdots n_{k+k_{0}}}{-\log
c_{k+1}\cdots c_{k+k_{0}}}<1
\]
for some $k_{0}\geq1,$ then any $E\in\mathcal{A}$ is uniformly disconnected.
If $E\in\mathcal{A}$ is uniformly disconnected and%
\[
\sup_{r^{\prime}<r_{0}r<r<r_{0}}\frac{\log\Phi_{\mathcal{A}}(r,r^{\prime}%
)}{\log\Phi_{\mathcal{B}}(r,r^{\prime})}<1
\]
for some $r_{0}<1,$ then $E\hookrightarrow F$ for any $F\in\mathcal{B}.$ If
$E\in\mathcal{A}$ and $F\in\mathcal{B}$ are uniformly disconnected, then $E$
and $F$ are quasi-Lipschitz equivalent if and only if
\[
\lim\limits_{r\rightarrow0}\frac{\log\Phi_{\mathcal{A}}(r)}{\log
\Phi_{\mathcal{B}}(r)}=1.
\]

\end{corollary}

\begin{example}
Let $J=[0,1],$ $c_{k}\equiv d_{k}\equiv1/5$ and $n_{k},m_{k}\in\{2,3\}$ for
all $k.$ Denote
\begin{align*}
a_{k,k^{\prime}}  &  =\#\{i:n_{i}=3 \mathrm{~with~ } k\leq i\leq k^{\prime
}\},\\
b_{k,k^{\prime}}  &  =\#\{i:m_{i}=3 \mathrm{~with~ } k\leq i\leq k^{\prime}\}.
\end{align*}
Let $E\in\mathcal{M(}J,\{n_{k}\}_{k},\{c_{k}\}_{k}\mathcal{)},F\in
\mathcal{M(}J,\{m_{k}\}_{k},\{d_{k}\}_{k}\mathcal{)}.$ Then $E,F$ are
uniformly disconnected. It follows from Corollary $\ref{C:Moran enbed}$ that
if there exist constants $k_{0}$ and $k_{1}$ such that
\[
0\leq a_{k,k+k_{0}}<b_{k,k+k_{0}}\text{ for all } k>k_{1},
\]
then $E\hookrightarrow F.$ Let
\[
a_{k}=\#\{i:n_{i}=3\text{ with }i\leq k\}, \quad b_{k}=\#\{i:m_{i} =3\text{
with }i\leq k\}.
\]
Using Corollary $\ref{C:Moran enbed}$ again, we obtain that $E$ and $F$ are
quasi-Lipschitz equivalent if and only if
\[
\lim_{k\rightarrow\infty}\frac{a_{k}\log3+(k-a_{k})\log2}{b_{k}\log
3+(k-b_{k})\log2}=1,
\]
which is equivalent to
\[
\lim_{k\rightarrow\infty}\frac{a_{k}+k}{b_{k}+k}=1.
\]

\end{example}

\bigskip

We would mention that this paper is quite different from the previous works,
e.g. \cite{Mattila Saarenen}, \cite{Deng juan}, \cite{Xi quasi} and \cite{Wang
Xi N}. For the fractals discussed in this paper, their Hausdorff dimensions
and packing dimensions need not be the same; they are more complicated than
Ahlfors--David regular sets as in \cite{Mattila Saarenen}, \cite{Deng juan},
\cite{Xi quasi} and \cite{Wang Xi N}. We notice that the main tool of this
paper is Moran set rather than self-similar set satisfying \textbf{OSC}.

The paper is organized as follows. In Section 2, we obtain the dimensions of
homogeneous sets. In Section 3, we show that Moran sets are homogeneous, and
we also give many homogeneous sets which are not Ahlfors--David regular. In
Section 4, we approximate the homogeneous sets by Moran sets. The bilipschitz
embedding and quasi-Lipschitz equivalence of homogeneous sets are discussed in
Sections 5 and 6 respectively.

\section{Dimensions of homogeneous fractals}

In this section, we will prove Proposition \ref{P:Prop of Homo}.

For a compact subset $A$ in any metric space, let $P(A,r)$ denote the greatest
number of disjoint $r$-balls with centers in $A,$ and $N(A,r)$ the smallest
number of $r$-balls needed to cover $A.$ We have%
\begin{equation}
N(A,2r)\leq P(A,r)\leq N(A,r/2)\text{ for any }r>0, \label{E:N<M}%
\end{equation}
please refer to Section 5.3 of \cite{Mattila}.

\begin{proof}
[Proof of Proposition \ref{P:Prop of Homo}]$\ $

For any $r\leq|A|,$ assume that $(\kappa_{A})^{n}|A|<r\leq(\kappa_{A}%
)^{n-1}|A|$ $(n\geq1)$; then
\[
\mu(B(x,(\kappa_{A})^{n}|A|))\leq\mu(B(x,r))\leq\mu(B(x,(\kappa_{A}%
)^{n-1}|A|)).
\]
Using (\ref{***}), we have for any $x\in A,$
\[
\frac{\mu(A)}{(\Delta_{A})^{n}}=\frac{\mu(B(x,|A|))}{(\Delta_{A})^{n}}\leq
\mu(B(x,r))\leq\frac{\mu(B(x,|A|))}{(\delta_{A})^{n-1}}=\frac{\mu(A)}%
{(\delta_{A})^{n-1}},
\]
which implies%
\begin{equation}
\frac{\log\delta_{A}}{-\log\kappa_{A}}\leq\liminf_{r\rightarrow0}\alpha
_{A}(r)\leq\limsup_{r\rightarrow0}\alpha_{A}(r)\leq\frac{\log\Delta_{A}}%
{-\log\kappa_{A}}. \label{E:alpha 0 infty}%
\end{equation}

(1) Fix $x^{\ast}\in A.$ For any $r>0,$ by (\ref{cc}) in Definition
\ref{D:Main}, we obtain%
\[
P(A,r)\cdot\lambda_{A}^{-1}\mu(B(x^{\ast},r))\leq\mu(A)\leq N(A,r)\cdot
\lambda_{A}\mu(B(x^{\ast},r)).
\]
Then by (\ref{E:N<M}), we have
\[
\frac{\mu(A)}{\lambda_{A}\mu(B(x^{\ast},r))}\leq N(A,r)\leq\frac{\lambda
_{A}\mu(A)}{\mu(B(x^{\ast},r/2))}.
\]
It follows from (\ref{***}) that $\mu$ is doubling, i.e., $\mu(B(x^{\ast
},r/2))\geq C\mu(B(x^{\ast},r))$ for some constant $C>0.$ Therefore,
\begin{equation}
f_{A}(r)=\frac{\log N(A,r)}{-\log r}\sim\alpha_{A}(x^{\ast},r).
\label{E:f/alpha}%
\end{equation}

(2) Using definitions of dimensions (see e.g. \cite{Cutler} and
\cite{Falconer}) and (\ref{E:f/alpha}), we have
\[
\dim_{H}A\leq\underline{\dim}_{B}A=\liminf_{r\rightarrow0}\alpha_{A}(r),\text{
}\dim_{P}A\leq\overline{\dim}_{B}A=\limsup_{r\rightarrow0}\alpha_{A}(r).
\]
It suffices to show that
\[
\dim_{H}A\geq\liminf_{r\rightarrow0}\alpha_{A}(r)\text{ and }\dim_{P}%
A\geq\limsup_{r\rightarrow0}\alpha_{A}(r).
\]

For any $0<s<\liminf_{r\rightarrow0}\alpha_{A}(r),$ there exists $r_{0}%
\in(0,1),$ such that for any $x\in A$ and $r\in(0,r_{0}],$%
\[
\alpha_{A}(x,r)=\frac{\log\mu(B(x,r))}{\log r}>s.
\]
Then for any subset $U\subset X$ with $A\cap U\neq\varnothing$ and $|U|\leq
r_{0},$
\[
\mu(U)\leq\mu(B(x,|U|))\leq|U|^{s}\text{ for any }x\in A\cap U.
\]
In a standard way, we get $\dim_{H}A\geq\liminf_{r\rightarrow0}\alpha_{A}(r).$

By the Corollary $3.20$(b) of \cite{Cutler}, we have $\dim_{P}A\geq
\limsup_{r\rightarrow0}\alpha_{A}(r)$ directly.

\end{proof}

\section{ Moran sets are homogeneous}

Given a Moran set $E\in\mathcal{M}(J,\{n_{k}\}_{k},\{c_{k}\}_{k})$ in
$\mathbb{R}^{d},$ for word $\sigma=i_{1}\cdots i_{k}\in\Omega_{k}$ with length
$k,$ write $J_{\sigma}=J_{i_{1}\cdots i_{k}},$ a basic element of order $k.$
Without loss of generality, for the proof of Proposition \ref{P:Mor is Homo},
we may assume that $|J|=1.$ Let $r_{k}=c_{1}\cdots c_{k}$ for all $k$, and let
$c_{1}\cdots c_{k-1}=n_{1}\cdots n_{k-1}=1$ for $k=1$.

Let $\mathcal{L}$ denote the Lebesgue measure on $\mathbb{R}^{d}.$ Write
$\text{\textrm{int}}(\cdot)$ the interior of set. Then
\begin{equation}
\mathcal{L}(\mathrm{int}(J_{\sigma}))=(r_{k})^{d}\mathcal{L}(\mathrm{int}%
(J))\label{test}%
\end{equation}
for $\sigma\in\Omega_{k}$ since $J_{\sigma}$ is geometrically similar to $J.$

Notice that the union $\bigcup\nolimits_{i_{k}=1}^{n_{k} }\mathrm{int}%
(J_{i_{1}\cdots i_{k-1}i_{k}})\subset\mathrm{int}(J_{i_{1}\cdots i_{k-1}})$ is
disjoint for any word $i_{1}\cdots i_{k-1}\in\Omega_{k-1}$, we have
\[
\sum\nolimits_{i_{k}=1}^{n_{k}}\mathcal{L}(\mathrm{int}(J_{i_{1}\cdots
i_{k-1}i_{k}}))\leq\mathcal{L}(\mathrm{int}(J_{i_{1}\cdots i_{k-1}})).
\]
Applying (\ref{test}) to the above formula, we have%
\begin{equation}
n_{k}c_{k}^{d}\leq1. \label{zz}%
\end{equation}
Applying $c_{\ast}=\inf_{k}c_{k}>0$ and $n_{k}\geq2$ to (\ref{zz}),\ we have%
\begin{equation}
c_{\ast}\leq c^{\ast}:=\sup_{k}c_{k}\leq\frac{1}{\sqrt[d]{2}}\ \text{and\ }%
2\leq n_{k}\leq c_{\ast}^{-d}. \label{dd}%
\end{equation}

\subsection{Moran measure}

\

We are going to construct a Borel probability measure $\mu$ on $\mathbb{R}%
^{d}$ with its support $\text{supp}\mu=E$ as in \cite{Cawley Mauldin}, which
is usually called the Moran measure.

Let $\Omega^{\infty}=\prod^{\infty}_{k=1}\{1,\cdots,n_{k}\}$ be a compact
metrizable space. For $w=w_{1}w_{2}\cdots\in\Omega^{\infty}$ and $k\geq1$, let
$w|_{k}=w_{1}\cdots w_{k}\in\Omega_{k}$. For $k\geq1$ and $\sigma\in\Omega
_{k}$, let $C_{\sigma}=\{w\in\Omega^{\infty}: w|_{k}=\sigma\}$, the cylinder
set determined by $\sigma$. Then there is a unique Borel probability measure
$\nu$ on $\Omega^{\infty}$ such that $\nu(C_{\sigma})=(n_{1}\cdots n_{k}%
)^{-1}$ for all $k\geq1$ and $\sigma\in\Omega_{k}$.

By (\ref{dd}), we notice that $r_{k}\rightarrow0$ as $k\rightarrow\infty$,
that is $\lim_{k\rightarrow\infty}|J_{w|_{k}}|=0. $ Thus there is a map
$f\colon\Omega^{\infty}\rightarrow\mathbb{R}^{d}$ with $f(\Omega^{\infty})=E$
defined by
\[
\{f(w)\}=\bigcap_{k=1}^{\infty}J_{w|_{k}} \text{ for each }w\in\Omega^{\infty
};
\]
and as $f(C_{\sigma})\subset J_{\sigma}$ for each $\sigma\in\Omega^{\ast
}=\bigcup^{\infty}_{k=0}\Omega_{k}$, the map $f$ is continuous. Now there is a
Borel probability measure $\mu$ on $\mathbb{R}^{d}$ defined by $\mu
(A)=\nu(f^{-1}(A))$ for all Borel set $A\subset\mathbb{R}^{d}$. Now
\begin{equation}
\label{E:mufnu}\mu(J_{\sigma})=\nu(f^{-1}(J_{\sigma}))\geq\nu(C_{\sigma
})=(n_{1}\cdots n_{k})^{-1}%
\end{equation}
for all $k\geq1$ and $\sigma\in\Omega_{k}$. From this it easily follows that
the support of $\mu$ is $E$.

Next, we give an estimation of the Moran measure.

\begin{lemma}
\label{L:Moran}There is a constant $C_{E}>1$ such that for any $x\in E$ and
$r_{k}<r\leq r_{k-1}$,
\[
(n_{1}\cdots n_{k})^{-1}\leq\mu(B(x,r))\leq C_{E}(n_{1}\cdots n_{k-1})^{-1}.
\]

\end{lemma}

\begin{proof}
Suppose $J_{\sigma}$ is a basic element of order $k$ containing $x$; since
$|J_{\sigma}|=r_{k},$ we have $J_{\sigma}\subset B(x,r).$ By (\ref{E:mufnu})
we have
\begin{equation}
\mu(B(x,r))\geq\mu(J_{\sigma})\geq(n_{1}\cdots n_{k})^{-1}.\label{xx}%
\end{equation}

Let $\Lambda_{x,r}=\{\sigma^{\prime}:\sigma^{\prime}\in\Omega_{k-1}$ and
$J_{\sigma^{\prime}}\cap B(x,r)\neq\varnothing\}.$ We will show that
$\#\Lambda_{x,r}\leq C_{E}$ for some constant $C_{E}>1$ independent of $x$ and
$r.$

Since \textrm{int}($J_{\sigma^{\prime}}$)$\cap$ \textrm{int}($J_{\sigma
^{^{\prime\prime}}}$)=$\varnothing$ for all distinct $\sigma^{\prime}$ and
$\sigma^{\prime\prime}$ in $\Lambda_{x,r}$, and
\[
\bigcup\nolimits_{\sigma^{\prime}\in\Lambda_{x,r}}\mathrm{int}(J_{\sigma
^{\prime}})\subset B(x,r+r_{k-1})\subset B(x,2r_{k-1}),
\]
we have
\begin{align*}
(\#\Lambda_{x,r})\mathcal{L}(\text{\textrm{int}}(J))(r_{k-1})^{d}  &
={\sum\nolimits_{\sigma^{\prime}\in\Lambda_{x,r}}}\mathcal{L}%
(\text{\textrm{int}}(J_{\sigma^{\prime}}))\\
&  =\mathcal{L}\left( \bigcup\nolimits_{\sigma^{\prime}\in\Lambda_{x,r}%
}\text{\textrm{int}}(J_{\sigma^{\prime}})\right) \\
&  \leq2^{d}(r_{k-1})^{d}\mathcal{L}(B(0,1)),
\end{align*}
which implies $\#\Lambda_{x,r}\leq\frac{2^{d}\mathcal{L}(B(0,1))}%
{\mathcal{L}(\text{\textrm{int}}(J))}=:C_{E}.$ Therefore,
\begin{align}
\label{yy}\mu(B(x,r))  &  =\nu(f^{-1}(B(x,r)))\nonumber\\
&  \leq\nu\{w\in\Omega^{\infty}:f(C_{w|_{k-1}})\cap B(x,r)\neq\varnothing\}\\
&  \leq{\displaystyle\sum\limits_{\sigma^{\prime}\in\Lambda_{x,r}}}
\nu(C_{\sigma^{\prime}})\leq C_{E}(n_{1}\cdots n_{k-1})^{-1}.\nonumber
\end{align}

Then this lemma follows from (\ref{xx}) and (\ref{yy}).
\end{proof}

\subsection{Proof of Proposition \ref{P:Mor is Homo}}

\

Using Lemma \ref{L:Moran}, we can prove that all Moran sets are\ homogeneous.

\begin{proof}
[Proof of Proposition \ref{P:Mor is Homo}]$\ $ Take $\lambda_{E}=C_{E}c_{\ast
}^{-d}.$\ For any $x_{1},x_{2}\in E,$ $r\in(0,|E|],$ if $r_{k}<r\leq r_{k-1}$
$(k\geq1),$ by Lemma \ref{L:Moran}, we have
\[
\lambda_{E}^{-1}\leq\frac{1}{C_{E}n_{k}}\leq\frac{\mu(B(x_{1},r))}{\mu
(B(x_{2},r))}\leq C_{E}n_{k}\leq\lambda_{E}.
\]

\medskip

Take $\kappa_{A}\in(0,1)$ small enough such that
\begin{equation}
\delta_{A}:=\frac{1}{C_{E}}\cdot2^{\frac{\log\kappa_{A}}{\log c_{\ast}}-2}>1.
\label{zzz}%
\end{equation}

Assume that $r_{k}<r\leq r_{k-1}$ and $r_{k^{\prime}}<\kappa_{A}r\leq
r_{k^{\prime}-1}$ with $k^{\prime}\geq k$. Then $k^{\prime}\geq k+1$ and
\begin{equation}
\label{tttt}\frac{n_{1}\cdots n_{k^{\prime}-1}}{C_{E}n_{1}\cdots n_{k}}%
\leq\frac{\mu(B(x,r))}{\mu(B(x,\kappa_{A}r))}\leq\frac{C_{E}n_{1}\cdots
n_{k^{\prime}}}{n_{1}\cdots n_{k-1}},
\end{equation}
where
%

\begin{equation}
\label{ttt}%
\begin{split}
\frac{C_{E}n_{1}\cdots n_{k^{\prime}}}{n_{1}\cdots n_{k-1}}  &  \leq
C_{E}(n_{k}\cdots n_{k^{\prime}})\leq C_{E}(c_{\ast}^{-d})^{k^{\prime}-k+1},\\
\frac{n_{1}\cdots n_{k^{\prime}-1}}{C_{E}n_{1}\cdots n_{k}}  &  \geq\frac
{1}{C_{E}}(n_{k+1}\cdots n_{k^{\prime}-1})\geq\frac{1}{C_{E}} \cdot
2^{k^{\prime}-k-1}.
\end{split}
\end{equation}

Now, we have%
\[
(c_{k}\cdots c_{k^{\prime}}=)\frac{r_{k^{\prime}}}{r_{k-1}}\leq\kappa_{A}%
\leq\frac{r_{k^{\prime}-1}}{r_{k}}(=c_{k+1}\cdots c_{k^{\prime}-1}),
\]
which implies
\[
(k^{\prime}-k+1)\log c_{\ast}\leq\log\kappa_{A}\leq(k^{\prime}-k-1)\log
c^{\ast},
\]
i.e.,
\begin{equation}
\frac{\log\kappa_{A}}{\log c_{\ast}}-1\leq k^{\prime}-k\leq\frac{\log
\kappa_{A}}{\log c^{\ast}}+1. \label{yyy}%
\end{equation}

Let $\Delta_{A}=C_{E}(c_{\ast}^{-d})^{\frac{\log\kappa_{A}}{\log c^{\ast}}%
+2}.$ Applying (\ref{zzz}) and (\ref{yyy}) to (\ref{tttt})-(\ref{ttt}), we
obtain
\[
\delta_{A}\leq\frac{\mu(B(x,r))}{\mu(B(x,\kappa_{A}r))}\leq\Delta_{A}.
\]

Lemma \ref{L:Moran} and (\ref{dd}) shows that we can take
\[
\alpha_{E}(r)=\frac{\log n_{1}\cdots n_{k}}{-\log c_{1}\cdots c_{k}}%
\]
whenever $r_{k}<r\leq r_{k-1}.$
\end{proof}

\subsection{Moran sets which are not Ahlfors--David regular}

\

For Moran set $E,$ it follows from Propositions \ref{P:Prop of Homo}%
-\ref{P:Mor is Homo} (also see \cite{Feng Wen Wu} and \cite{Wen}) that
\begin{equation}
\dim_{H}E=\liminf_{k\rightarrow\infty}\frac{\log n_{1}\cdots n_{k}}{-\log
c_{1}\cdots c_{k}},\text{ }\dim_{P}E=\limsup_{k\rightarrow\infty}\frac{\log
n_{1}\cdots n_{k}}{-\log c_{1}\cdots c_{k}}. \label{dim}%
\end{equation}
Since $\dim_{H}F=\dim_{P}F$ for any Ahlfors--David regular set $F,$ we have

\begin{proposition}
\label{P;homo not regular}If $\liminf_{k\rightarrow\infty}\frac{\log
n_{1}\cdots n_{k}}{-\log c_{1}\cdots c_{k}}<\limsup_{k\rightarrow\infty}%
\frac{\log n_{1}\cdots n_{k}}{-\log c_{1}\cdots c_{k}},$ then $E$ is not
Ahlfors--David regular for any $E\in\mathcal{M}(J,\{n_{k}\}_{k},\{c_{k}%
\}_{k}).$
\end{proposition}

\begin{remark}
The above result shows that we can find lots of homogeneous sets which are not
Ahlfors--David regular.
\end{remark}

The following example shows that a homogeneous set $E\ $with $\dim_{H}%
E=\dim_{P}E$ need not be Ahlfors--David regular.

\begin{example}
\label{EX:k+1/k+2 copy(1)}Let $n_{k}\equiv2$ and $c_{k}=\frac{k+1}{2(k+2)}$
for all $k\geq1.$ Denote $J=[0,1],$ and let $J_{1}=[0,c_{1}],$\ $J_{2}%
=[1-c_{1},1]$.\ Inductively, if the interval $J_{i_{1}\cdots i_{k}}%
=[c_{i_{1}\cdots i_{k}},d_{i_{1}\cdots i_{k}}]$ have been defined$,$ we define
its subintervals $J_{i_{1}\cdots i_{k}1}=[c_{i_{1}\cdots i_{k}},c_{i_{1}\cdots
i_{k}}+c_{k}|J_{i_{1}\cdots i_{k}}|]$ and $J_{i_{1}\cdots i_{k}2}%
=[d_{i_{1}\cdots i_{k}}-c_{k}|J_{i_{1}\cdots i_{k}}|,d_{i_{1}\cdots i_{k}}]$.
As above, we get a Moran set $E.$ Using $(\ref{dim})$, we have $\dim_{H}%
E=\dim_{P}E=1.$ Notice that $\mathcal{H}^{1}(E)=\mathcal{L(}E\mathcal{)}=0,$
where $\mathcal{L}$ is the Lebesgue measure. If $E$ is Ahlfors--David
$1$-regular, then $\mathcal{H}^{1}(E)>0$, which is a contradiction. That means
$E$ is not Ahlfors--David regular.

\bigskip
\end{example}

\section{Approximating by Moran sets}

\subsection{Bilipschitz image of homogeneous set}

\

Under bilipschitz mapping, the homogeneous property will be preserved$.$

\begin{lemma}
\label{L:Bili of homo}Suppose $A(\subset X)$ is a homogeneous set. If $A$ is
bilipschitz equivalent to $B(\subset Y)$, then $B$ is also homogeneous and
$\chi(A,B)=0.$
\end{lemma}

\begin{proof}
Assume that $f$ is the bilipschitz map from $A$ onto $B$ with bilipschitz
constant $L\geq1,$ and $\mu$ is the corresponding measure supported on $A.$ We
define the image measure $\nu$ on $B$ with $\nu(F)=\mu(f^{-1}(F))$ for any
Borel subset $F\subset B.$ It is clear that $\nu$ is a Borel measure supported
on $B.$

Without loss of generality, we may assume that $A=X$ and $B=Y,$ the whole
metric spaces. For any $y\in B,$ $0<r\leq|B|,$ we have%
\begin{equation}
B(f^{-1}(y),r/L)\subset f^{-1}(B(y,r))\subset B(f^{-1}(y),Lr);
\label{E:Lip image}%
\end{equation}
then
\begin{equation}
\mu(B(f^{-1}(y),r/L))\leq\nu(B(y,r))\leq\mu(B(f^{-1}(y),Lr)).
\label{E:mu nu mu}%
\end{equation}

Using Definition \ref{D:Main} and (\ref{E:mu nu mu})$,$ for any $y_{1}%
,y_{2}\in B$ and $r\leq|A|/L,$ we have
\[
\frac{\nu(B(y_{1},r))}{\nu(B(y_{2},r))}\leq\frac{\mu(B(f^{-1}(y_{1}),Lr))}%
{\mu(B(f^{-1}(y_{2}),r/L))}\leq\lambda_{A}\frac{\mu(B(f^{-1}(y_{1}),Lr))}%
{\mu(B(f^{-1}(y_{1}),r/L))}\leq\lambda_{A}(\Delta_{A})^{n},
\]
where $(\kappa_{A})^{n}\leq L^{-2}<(\kappa_{A})^{n-1}$ for some integer
$n\geq0.$ Then $n-1< \frac{-2\log L}{\log\kappa_{A}}\leq n,$ and thus
$n\leq1-\frac{2\log L}{\log\kappa_{A}}.$ Therefore, for all $y_{1},y_{2}\in B$
and $r\leq|A|/L,$%
\[
\frac{\nu(B(y_{1},r))}{\nu(B(y_{2},r))}\leq\lambda_{A}(\Delta_{A})^{n}%
\leq\lambda:=\lambda_{A}\left(  \Delta_{A}\right)  ^{1-\frac{2\log L}%
{\log\kappa_{A}}},
\]
which implies
\[
\lambda^{-1}\leq\frac{\nu(B(y_{1},r))}{\nu(B(y_{2},r))}\leq\lambda\text{
}\text{for any}\text{ }y_{1},y_{2}\in B\text{, }r\leq|A|/L.
\]

Fix a point $y^{\ast}\in B$ and let $\lambda_{B}=\frac{\mu(A)}{\nu(B(y^{\ast
},|A|/L))}\cdot\lambda\geq\lambda.$

Given any $r\in\lbrack|A|/L,|B|],$ we have $\nu(B(y_{1},r))\leq\mu(A)$ and
$\nu(B(y_{2},r))\geq\nu(B(y_{2},|A|/L))\geq\lambda^{-1}\nu(B(y^{\ast
},|A|/L)),$ which implies
\begin{equation}
\lambda_{B}^{-1}\leq\frac{\nu(B(y_{1},r))}{\nu(B(y_{2},r))}\leq\lambda
_{B}\text{ }\text{for any}\text{ }y_{1},y_{2}\in B\text{, }r\leq|B|.
\label{B1}%
\end{equation}

\medskip

Let $\kappa_{B}=\kappa_{A}/L^{2}.$ Using (\ref{E:mu nu mu}), for $r\leq
|B|\leq|A|\cdot L,$ for any $x\in A$ we have
\[
\frac{\nu(B(f(x),r))}{\nu(B(f(x),\kappa_{B}r))}\geq\frac{\mu(B(x,r/L))}%
{\mu(B(x,\kappa_{A}r/L))}\geq\delta_{A},
\]
since $r/L\leq|A|.$

On the other hand, if $r\leq|A|/L,$ then
\[
\frac{\nu(B(f(x),r))}{\nu(B(f(x),\kappa_{B}r))}\leq\frac{\mu(B(x,rL))}%
{\mu(B(x,\kappa_{A}r/L^{3}))}\leq(\Delta_{A})^{m},
\]
where $(\kappa_{A})^{m}\leq\kappa_{A}L^{-4}<(\kappa_{A})^{m-1}$ for some
integer $m\geq1.$ Then $(\kappa_{A})^{m-1}\leq L^{-4}<(\kappa_{A})^{m-2},$
i.e., $m-2\leq-\frac{4\log L}{\log\kappa_{A}}\leq m-1.$ Therefore,
\[
\frac{\nu(B(f(x),r))}{\nu(B(f(x),\kappa_{B}r))}\leq(\Delta_{A})^{2-\frac{4\log
L}{\log\kappa_{A}}}\text{ for any }r\leq|A|/L.
\]
Let $\Delta_{B}=\max((\Delta_{A})^{2-\frac{4\log L}{\log\kappa_{A}}}%
,\frac{\lambda_{B}\mu(A)}{\nu(B(y^{\ast},\kappa_{B}|A|/L))})$. Then we have%
\[
\frac{\nu(B(f(x),r))}{\nu(B(f(x),\kappa_{B}r))}\leq\Delta_{B}\text{ for all
}r\leq|B|.
\]
Therefore, for any $x\in A$ and $r\leq|B|,$
\begin{equation}
\delta_{A}\leq\frac{\nu(B(f(x),r))}{\nu(B(f(x),\kappa_{B}r))}\leq\Delta_{B}.
\label{B2}%
\end{equation}
It follows from (\ref{B1}) and (\ref{B2}) that $B$ is also homogeneous.

\medskip

Using (\ref{E:Lip image}), we have
\begin{equation}
N(B,Lr)\leq N(A,r)\leq N(B,L^{-1}r). \label{E:Lip Num}%
\end{equation}
It follows from (\ref{behavior}), (\ref{E:Lip Num}) and the fact that $A$ is a
doubling metric space that $\chi(A,B)=0.$
\end{proof}

\begin{proof}
[Proof of Proposition \ref{p:mei}]$\ $

In fact, suppose $\sup_{r<r_{0}}\alpha_{A}(r)<\infty$ for some $r_{0}$ small
enough$.$ We note that $\varphi(x)=e^{x}-1$ is continuous at $0$, then for fixed $\varepsilon>0$ small enough, there exists
$\delta>0$ such that if $\chi(A,B)=\limsup_{r\rightarrow0}\left\vert
\log\frac{\alpha_{B}(r)}{\alpha_{A}(r)}\right\vert <\delta,$ then
$|\frac{\alpha_{B}(r)}{\alpha_{A}(r)}-1|=|\varphi(\log\frac{\alpha_{B}(r)}{\alpha_{A}(r)})|<\varepsilon/(\sup_{r<r_{0}}\alpha
_{A}(r))$ for all $r<r_{1}\,$where $r_{1}<r_{0}$ is a constant. Hence
\[
|\alpha_{B}(r)-\alpha_{A}(r)|<\varepsilon\text{ for all }r<r_{1},
\]
and thus $|\overline{\lim}_{r\rightarrow0}\alpha_{B}(r)-\overline{\lim
}_{r\rightarrow0}\alpha_{A}(r)|,$ $|\underline{\lim}_{r\rightarrow0}\alpha
_{B}(r)-\underline{\lim}_{r\rightarrow0}\alpha_{A}(r)|<\varepsilon.$ It
follows from (1) of Proposition \ref{P:Prop of Homo} that $|\dim_{P}B-\dim
_{P}A|,|\dim_{H}B-\dim_{H}A|<\varepsilon.$
\end{proof}

\subsection{Proof of the approximation theorem}

\

For homogeneous sets, we can approximate them by their subsets which are
bilipschitz images of Moran sets in Euclidean spaces.

\begin{proof}
[Proof of Theorem \ref{T:E Emb A}]\

We can prove Theorem \ref{T:E Emb A} in three steps:

\begin{enumerate}
\item For any $\varepsilon>0,$ choose $\eta$ small enough and construct a
subset $A(\eta)$ of $A,$ such that $\mathrm{d}_{H}(A(\eta),A)<\varepsilon.$

\item Corresponding to $A(\eta),$ construct a Moran set $E(\eta)$ in
$\mathbb{R}^{d}$ for some $d\in\mathbb{N}.$ Show that the natural bijection
between $A(\eta)$ and $E(\eta)$ is a bilipschitz map.

\item Verify that $\chi(E(\eta),A)=\chi(A(\eta),A)<\varepsilon.$
\end{enumerate}

Without loss of generality, assume that $A$ is homogeneous with $|A|=1.$ Let
$\underline{\mu}(r)=\inf_{x\in A}\mu(B(x,r))$ and $\overline{\mu}%
(r)=\sup_{x\in A}\mu(B(x,r)).$ Fix a point $x^{\ast}\in A;$ then using
(\ref{cc}) and (\ref{***}), we have
\[
\frac{\underline{\mu}(r/2)}{\overline{\mu}(2r^{\prime})}\geq\frac{1}%
{(\lambda_{A})^{2}}\frac{\mu(B(x^{\ast},r/2))}{\mu(B(x^{\ast},2r^{\prime}%
))}\geq\frac{1}{(\lambda_{A})^{2}}(\delta_{A})^{n},
\]
where $\frac{2r^{\prime}}{r/2}\leq(\kappa_{A})^{n}$ for some integer $n.$
Taking $n$ large enough, we have
\begin{equation}
\frac{\underline{\mu}(r/2)}{\overline{\mu}(2r^{\prime})}\geq2\ \text{if }%
\frac{r^{\prime}}{r}\leq\eta_{0} \label{ty1}%
\end{equation}
for some constant $\eta_{0}.$ By Definition \ref{D:Main}$,$ there exists a
constant $C_{0}\in(0,1)$ such that for any $r,r^{\prime}$ with $\frac
{r^{\prime}}{r}\leq\eta_{0}$,%
\begin{equation}
C_{0}\frac{\overline{\mu}(r)}{\overline{\mu}(r^{\prime})}\leq\left\lfloor
\frac{\underline{\mu}(r/2)}{\overline{\mu}(2r^{\prime})}\right\rfloor
\leq\frac{\overline{\mu}(r)}{\overline{\mu}(r^{\prime})},
\label{E:mu/mu mu/mu}%
\end{equation}
where $\left\lfloor \cdot\right\rfloor $ denotes the integral part of number.

\textbf{Step 1}. Let $\varepsilon>0$. Fix so large an integer $n$ that
\begin{equation}
\eta_{1}:=(\kappa_{A})^{n}<\min(\eta_{0},\frac{1}{4},\frac{\varepsilon}%
{3})\text{ and }\left\vert \frac{\log C_{0}}{n\log\delta_{A}}\right\vert
\leq\frac{\varepsilon}{2}. \label{txt2}%
\end{equation}
Now choose $\eta>0$ with $\eta\leq\eta_{1}$. Then $\overline{\mu}(\eta
^{k})\leq(\delta_{A})^{-nk}\overline{\mu}(1)$ for each $k\geq1$, and thus
\begin{equation}
\limsup_{k\rightarrow\infty}\left\vert \frac{(k-1)\log C_{0}}{\log
\overline{\mu}(\eta^{k})}\right\vert \leq\left\vert \frac{\log C_{0}}%
{n\log\delta_{A}}\right\vert \leq\frac{\varepsilon}{2} . \label{E:eta2}%
\end{equation}

For all $k\geq2,$ let
\begin{equation}
\label{E:nk geq 2}n_{k}=\left\lfloor \frac{\underline{\mu}(\eta^{k-1}%
/2)}{\overline{\mu}(2\eta^{k})}\right\rfloor \geq2,
\end{equation}
due to (\ref{ty1}) as $\eta<\eta_{0}.$

We begin to construct the $A(\eta).$ In the first step of the construction, we
get a maximal number $P_{A}=P(A,\eta)$ of disjoint $\eta$-balls $\{B(x_{i}%
,\eta)\}_{i=1}^{P_{A}}$ with centers in $A$. For a small enough $\eta,$ let
\[
n_{1}=P_{A}\geq2.
\]

Given $\{n_{k}\}_{k}$ as above, let $\Omega^{\infty}$ denote the collection of
all infinite sequences $i_{1}\cdots i_{k}\cdots$ with $i_{1}\cdots i_{k}%
\in\Omega_{k}$ for every $k\geq1.$

For $k\geq2$, inductively assume that for $k-1$, we have obtained a family of
disjoint balls $\{B(x_{i_{1}\cdots i_{k-1}},\eta^{k-1})\}_{i_{1}\cdots
i_{k-1}\in\Omega_{k-1}}.$ We will find $\{B(x_{i_{1}\cdots i_{k-1}i_{k}}%
,\eta^{k})\}_{i_{1}\cdots i_{k-1}i_{k}\in\Omega_{k}}$ satisfying for every
$i_{1}\cdots i_{k-1}\in\Omega_{k-1},$

\begin{itemize}
\item $x_{i_{1}\cdots i_{k-1}i_{k}}\in B(x_{i_{1}\cdots i_{k-1}},\eta
^{k-1}/2)\cap A$ for all $1\leq i_{k}\leq n_{k}$;

\item $B(x_{i_{1}\cdots i_{k-1}i_{k}},\eta^{k})\cap B(x_{i_{1}\cdots
i_{k-1}j_{k}},\eta^{k})=\varnothing$ for all $i_{k}\neq j_{k}.$
\end{itemize}

In fact, fixing a sequence $i_{1}\cdots i_{k-1}\in\Omega_{k-1},$ we take a
maximal number $P_{i_{1}\cdots i_{k-1}}$ of disjoint $\eta^{k}$-balls with
centers in $B(x_{i_{1}\cdots i_{k-1}},\eta^{k-1}/2)\cap A.$ We will estimate
$P_{i_{1}\cdots i_{k-1}}$. At first, since $\eta<\frac{1}{4}$ by
(\ref{txt2})$,$ for every $\eta^{k}$-ball $B(x,\eta^{k})$ as above, we have%
\begin{equation}
B(x,\eta^{k})\subset B(x_{i_{1}\cdots i_{k-1}},\eta^{k-1}/2+\eta^{k})\subset
B(x_{i_{1}\cdots i_{k-1}},\frac{3}{4}\eta^{k-1}). \label{E:eta 3/4 eta}%
\end{equation}
Since these $P_{i_{1}\cdots i_{k-1}}$ disjoint $\eta^{k}$-balls are contained
in $B(x_{i_{1}\cdots i_{k-1}},\eta^{k-1})$, we have%
\begin{equation}
P_{i_{1}\cdots i_{k-1}}\leq\frac{\overline{\mu}(\eta^{k-1})}{\underline{\mu
}(\eta^{k})}. \label{E:M<mu/mu}%
\end{equation}
On the other hand, by (\ref{E:N<M}), $B(x_{i_{1}\cdots i_{k-1}},\eta
^{k-1}/2)\cap A$ can be covered by $P_{i_{1}\cdots i_{k-1}}$ balls of radius
$2\eta^{k},$ that means $\underline{\mu}(\eta^{k-1}/2)\leq P_{i_{1}\cdots
i_{k-1}}\cdot\overline{\mu}(2\eta^{k}),$ i.e.,
\[
P_{i_{1}\cdots i_{k-1}}\geq\frac{\underline{\mu}(\eta^{k-1}/2)}{\overline{\mu
}(2\eta^{k})}\geq n_{k}.
\]
Hence we can take $n_{k}$ disjoint $\eta^{k}$-balls with their centers in
$B(x_{i_{1}\cdots i_{k-1}},\eta^{k-1}/2)\cap A.$ Denote their centers by
$\{x_{i_{1}\cdots i_{k-1}i_{k}}\}_{i_{k}=1}^{n_{k}}.$

We define
\begin{equation}
A(\eta)=\bigcap\nolimits_{k\geq1}\bigcup\nolimits_{i_{1}\cdots i_{k}\in
\Omega_{k}}B(x_{i_{1}\cdots i_{k}},\eta^{k})\subset A.
\end{equation}
For any $i_{1}\cdots i_{k}\cdots\in\Omega^{\infty},$ let $x_{i_{1}\cdots
i_{k}\cdots}\in A(\eta)$ be such that
\begin{equation}
\{x_{i_{1}\cdots i_{k}\cdots}\}=\bigcap\nolimits_{k\geq1}B(x_{i_{1}\cdots
i_{k}},\eta^{k}).
\end{equation}

Since in the first step of the construction of $A(\eta)$, we get the maximal
number $P_{A}$ of disjoint $\eta$-balls $\{B(x_{i},\eta)\}_{i=1}^{P_{A}}$ with
centers in $A$, it follows that $A$ can be covered by $P_{A}$ balls
$\{B(x_{i},2\eta)\}_{i=1}^{P_{A}}$. Therefore%
\[
\mathrm{d}_{H}(A(\eta),A)\leq\mathrm{d}_{H}(\{x_{i}\}_{i=1}^{P_{A}%
},A)+\mathrm{d}_{H}(\{x_{i}\}_{i=1}^{P_{A}},A(\eta))\leq2\eta+\eta
<\varepsilon.
\]

\textbf{Step 2}. For the $\eta$ given, by Definition \ref{D:Main} from
(\ref{E:M<mu/mu}) it follows that $\{n_{k}\}_{k\geq1}$ is bounded. Then taking
$d$ large enough, we can construct a Moran set $E(\eta)$\ in $\mathbb{R}^{d}$
such that $E(\eta)\in\mathcal{M}( J,\{n_{k}\},\{c_{k}\})$ with $J=B(0,\frac
{1}{2}),$ $n_{k}$ defined above, and $\ c_{k}\equiv\eta$ for all $k\geq1$ such
that there is a constant $c>0$ for which
\[
\mathrm{d}(J_{i_{1}\cdots i_{k-1}i_{k}},J_{i_{1}\cdots i_{k-1}j_{k}})\geq
c\eta^{k}\text{ for all }i_{k}\neq j_{k}.
\]

For any $i_{1}\cdots i_{k}\cdots\in\Omega^{\infty},$ let $y_{i_{1}\cdots
i_{k}\cdots}\in E(\eta)$ be such that
\begin{equation}
\{y_{i_{1}\cdots i_{k}\cdots}\}={\bigcap\nolimits_{k}}J_{i_{1}\cdots i_{k}}.
\end{equation}

Naturally, we obtain a bijection $f$ from $E(\eta)$ to $A(\eta)$ such that
\begin{equation}
f(y_{i_{1}\cdots i_{k}\cdots})=x_{i_{1}\cdots i_{k}\cdots}.\text{ }%
\forall\text{ }i_{1}\cdots i_{k}\cdots\in\Omega^{\infty}.
\end{equation}
It suffices to show that $f$ is bilipschitz. In fact, for distinct points
$y^{\prime}=y_{i_{1}\cdots i_{k-1}i_{k}\cdots}$ and $y^{\prime\prime}%
=y_{i_{1}\cdots i_{k-1}j_{k}\cdots}$ with $i_{k}\neq j_{k}$ $(k\geq1),$ we
have
\begin{equation}
c\eta^{k}\leq\mathrm{d}(J_{i_{1}\cdots i_{k-1}i_{k}},J_{i_{1}\cdots
i_{k-1}j_{k}})\leq|y^{\prime}-y^{\prime\prime}|\leq|J_{i_{1}\cdots i_{k-1}%
}|=\eta^{k-1}. \label{t9}%
\end{equation}
On the other hand, $B(x_{i_{1}\cdots i_{k-1}i_{k}},\eta^{k})$ and
$B(x_{i_{1}\cdots i_{k-1}j_{k}},\eta^{k})$ are disjoint and
\[
x^{\prime}=x_{i_{1}\cdots i_{k-1}i_{k}\cdots}\in B(x_{i_{1}\cdots i_{k-1}%
i_{k}},\frac{3}{4}\eta^{k}),\text{ }x^{\prime\prime}=x_{i_{1}\cdots
i_{k-1}j_{k}\cdots}\in B(x_{i_{1}\cdots i_{k-1}j_{k}},\frac{3}{4}\eta^{k}),
\]
due to (\ref{E:eta 3/4 eta}); therefore,
\begin{equation}
\frac{1}{4}\eta^{k}\leq\mathrm{d}_{X}(x^{\prime},x^{\prime\prime}%
)\leq|B(x_{i_{1}\cdots i_{k-1}},\frac{3}{4}\eta^{k-1})|\leq\frac{3}{2}%
\eta^{k-1}. \label{t99}%
\end{equation}
It follows from (\ref{t9}) and (\ref{t99}) that $f$ is bilipschitz.

\textbf{Step 3}. For the Moran set $E(\eta)\in\mathcal{M}(J,\{n_{k}%
\}_{k},\{c_{k}\}_{k}),$ $J=B(0,1/2)$ with $|J|=1.$ Using Proposition
\ref{P:Mor is Homo}, we can take
\[
\alpha_{E(\eta)}(r)=\frac{\log n_{1}\cdots n_{k}}{-k\log\eta}\text{ for }%
\eta^{k}<r\leq\eta^{k-1},
\]
where $C_{0}\frac{\overline{\mu}(\eta^{k-1})}{\overline{\mu}(\eta^{k})}\leq
n_{k}\leq\frac{\overline{\mu}(\eta^{k-1})}{\overline{\mu}(\eta^{k})}$ for
$k\geq2$ due to (\ref{E:mu/mu mu/mu}), which implies%
\[
\frac{\log\overline{\mu}(\eta^{k})}{k\log\eta}-\left(  \frac{\log
n_{1}\overline{\mu}(\eta)}{k\log\eta}+\frac{(k-1)\log C_{0}}{k\log\eta
}\right)  \leq\alpha_{E(\eta)}(r)\leq\frac{\log\overline{\mu}(\eta^{k})}%
{k\log\eta}-\frac{\log n_{1}\overline{\mu}(\eta)}{k\log\eta}.
\]
Using (\ref{cc}) in Definition \ref{D:Main}$,$ for the homogeneous set $A$ we
have%
\begin{equation}
\frac{\log\overline{\mu}(\eta^{k-1})}{k\log\eta}\leq\alpha_{A}(x_{A}%
,r)\leq\frac{\log\underline{\mu}(\eta^{k})}{(k-1)\log\eta}\leq\frac
{\log\overline{\mu}(\eta^{k})-\log\lambda_{A}}{(k-1)\log\eta}.
\label{E:alpha k log eta}%
\end{equation}

It follows from (\ref{cc}) and (\ref{***}) that $\overline{\mu}(\eta
^{k-1})\geq\overline{\mu}(\eta^{k})\geq\varsigma\overline{\mu}(\eta^{k-1})$
for some constant $\varsigma>0,$ which implies
\begin{equation}
\frac{\log\overline{\mu}(\eta^{k-1})}{k\log\eta}-\frac{\log\overline{\mu}%
(\eta^{k})}{k\log\eta},\text{ }\frac{\log\overline{\mu}(\eta^{k})}%
{(k-1)\log\eta}-\frac{\log\overline{\mu}(\eta^{k})}{k\log\eta}=O\left(
\frac{1}{k\log\eta}\right)  . \label{txt}%
\end{equation}
By (\ref{E:alpha k log eta}) and (\ref{txt}), we can take a function
$\alpha_{A}(r)\sim\alpha_{A}(x_{A},r)$ defined by%
\[
\alpha_{A}(r)=\frac{\log\overline{\mu}(\eta^{k})}{k\log\eta}\text{ for }%
\eta^{k}<r\leq\eta^{k-1}.
\]
Using the inequality $|\log t|\leq\frac{3}{2}|t-1|$ for all $|t-1|\leq1/3$,
(\ref{E:eta2}) and Lemma \ref{L:Bili of homo}, we have, assuming
$\varepsilon/2\leq1/3$ as we may, that
\begin{align*}
&  \ \ \ \ \chi(A(\eta),A)=\chi(E(\eta),A)\\
&  =\limsup_{r\rightarrow0}\left\vert \log\frac{\alpha_{E(\eta)}(r)}%
{\alpha_{A}(r)}\right\vert \leq\frac{3}{2}\limsup_{r\rightarrow0}\left\vert
\frac{\alpha_{E(\eta)}(r)}{\alpha_{A}(r)}-1\right\vert \\
&  \leq\frac{3}{2}\limsup_{k\rightarrow\infty}\left\vert \frac{\log
n_{1}\overline{\mu}(\eta)}{\log\overline{\mu}(\eta^{k})}\right\vert +\frac
{3}{2}\limsup_{k\rightarrow\infty}\left\vert \frac{(k-1)\log C_{0}}%
{\log\overline{\mu}(\eta^{k})}\right\vert \\
&  \leq0+\frac{3}{2}\cdot\frac{\varepsilon}{2}<\varepsilon,
\end{align*}

In particular, if $A$ is a homogeneous set in $\mathbb{R}^{d},$ since any two
balls in $\mathbb{R}^{d}$ are geometrically similar, the above construction
shows that $A(\eta)$ is a Moran set. Take $f=id$ and $F=A(\eta)\subset A.$
Furthermore, using Proposition \ref{p:mei} we can approximate $A$ by Moran sets simultaneously
in three aspects: Hausdorff metric, Hausdorff dimension and
packing dimension.
\end{proof}

\section{Bilipschitz embedding of homogeneous sets}

\subsection{Necessary condition of bilipschitz embedding}

\

As shown in \cite{Deng juan}, a self-similar set satisfying {\textbf{SSC}} can
be bilipschitz embedded into any self-similar set with higher dimension.

However, for homogeneous fractals, we need the following new necessary
condition (Lemma \ref{L:blip}): if $A\hookrightarrow B,$ then
\begin{equation}
\frac{\mu(B(x,r))}{\mu(B(x,r^{\prime}))}\leq C\frac{\nu(B(y,r))}%
{\nu(B(y,r^{\prime}))}\text{ for all }r^{\prime}<r\leq\min(|A|,|B|),
\label{zip}%
\end{equation}
where $C$ is a constant.

\begin{proof}
[Proof of Lemma \ref{L:blip}]$\ $

Suppose that there is an injection $f\colon(A,$d$_{A})\rightarrow(B,$d$_{B})$
and a constant $L\geq1$ such that for all $x_{1},x_{2}\in A,$%
\[
\text{d}_{A}(x_{1},x_{2})/L\leq\text{d}_{B}(f(x_{1}),f(x_{2}))\leq
L\text{d}_{A}(x_{1},x_{2}).
\]

Given positive quantities $\{\theta_{\lambda}\}_{\lambda}$ and $\{\vartheta
_{\lambda}\}_{\lambda}$ with parameter $\lambda,$ we say that they are
comparable and denote $\theta_{\lambda}\asymp\vartheta_{\lambda},$ if there is
a constant $\rho$ independent of $\lambda$ such that
\[
\rho^{-1}\leq\frac{\theta_{\lambda}}{\vartheta_{\lambda}}\leq\rho.
\]

For any subset $\mathcal{C}$ of $A,$ let $K_{A}(\mathcal{C},r)=\max\{n:$ there
are distinct points $\{x_{i}\}_{i=1}^{n}$ of $\mathcal{C}$ such that
$\min_{i\neq j}$d$_{A}(x_{i},x_{j})\geq r\}.$ Therefore, for any $r^{\prime
}<r$,
\begin{equation}
K_{A}(B(x,r),r^{\prime})\leq K_{B}(B(f(x),Lr),r^{\prime}/L). \label{zip1}%
\end{equation}
Using Definition \ref{D:Main}, as in the proof of Proposition
\ref{P:Prop of Homo}, we obtain that
\begin{equation}
K_{A}(B(x,r),r^{\prime})\asymp P_{A}(B(x,r),r^{\prime})\asymp N_{A}%
(B(x,r),r^{\prime})\asymp\frac{\mu(B(x,r))}{\mu(B(x,r^{\prime}))},
\label{zip2}%
\end{equation}
where $P_{A}(\mathcal{C},r)=\max\{n:$ there are $n$ disjoint $r$-balls with
centers in $\mathcal{C}\}$ and $N_{A}(\mathcal{C},r)=\min\{n:$ there are $n$
$r$-balls covering $\mathcal{C}\}.$ Note that this result depends heavily on the fact that $A$ is a doubling metric space.

In the same way, we obtain that for any $y\in B,$%
\begin{equation}
K_{B}(B(f(x),Lr),r^{\prime}/L)\asymp\frac{\nu(B(f(x),Lr))}{\nu
(B(f(x),r^{\prime}/L))}\asymp\frac{\nu(B(f(x),r))}{\nu(B(f(x),r^{\prime}%
))}\asymp\frac{\nu(B(y,r))}{\nu(B(y,r^{\prime}))}. \label{zip3}%
\end{equation}

Thus (\ref{zip}) follows from (\ref{zip1})--(\ref{zip3}).

By (\ref{zip}), we have
\[
\sup\limits_{r^{\prime}<r_{0}r<r<r_{0}}\left\vert \frac{\log\mu(B(x,r))-\log
\mu(B(x,r^{\prime}))}{\log\nu(B(y,r))-\log\nu(B(y,r^{\prime}))}\right\vert
\leq1+\left\vert \frac{\log C}{\log\nu(B(y,r))-\log\nu(B(y,r^{\prime}%
))}\right\vert ,
\]
where $C$ is an independent constant. Taking $r_{0}$ small enough,
$\nu(B(y,r))/\nu(B(y,r^{\prime}))$ is so large that $\sup\limits_{r^{\prime
}<r_{0}r<r<r_{0}}\left\vert \frac{\log C}{\log\nu(B(y,r))/\nu(B(y,r^{\prime
}))}\right\vert $ is close to $0$.

On the other hand, $\left\vert \alpha_{A}(r)\log r-\log\mu(B(x,r))\right\vert
,\left\vert \alpha_{B}(r)\log r-\log\nu(B(y,r))\right\vert \leq C_{1}$ for
some $C_{1}$ due to (\ref{E:equ}), and $\log\mu(B(x,r))/\mu(B(x,r^{\prime
})),\log\nu(B(y,r))/\nu(B(y,r^{\prime}))$ are arbitrarily large when $r_{0}$
is small enough. Thus
\begin{equation}
\sup\limits_{r^{\prime}<r_{0}r<r<r_{0}}\left\vert \frac{\alpha_{A}(r)\log
r-\alpha_{A}(r^{\prime})\log r^{\prime}}{\alpha_{B}(r)\log r-\alpha
_{B}(r^{\prime})\log r^{\prime}}\right\vert \leq1+\varepsilon(r_{0}),
\end{equation}
with $\varepsilon(r_{0})\downarrow0$ as $r_{0}\downarrow0.$
\end{proof}

Now we will construct Moran set $B$ with number $t$ such that for any
Ahlfors--David regular set $A$ satisfying $t<\dim_{H}A<\dim_{H}B,$ the
inequality (\ref{zip}) fails.

\begin{proof}
[Proof of Proposition \ref{P:A<B}]$\ $

Let $t=\log3/\log6,$ $c_{k}\equiv1/6,$ $k_{m}=m^{3}$ and $t_{m}=k_{m}+m$ for
all $m.$ We take%
\[
n_{k}=\left\{
\begin{array}
[c]{ll}%
3 & \text{if }k\in\lbrack k_{m}+1,t_{m}]\text{ for some }m,\\
5 & \text{otherwise.}%
\end{array}
\right.
\]
Let $B\in\mathcal{M}([0,1],\{n_{k}\}_{k},\{c_{k}\}_{k}).$ Then it follows from
Propositions \ref{P:Prop of Homo} and \ref{P:Mor is Homo} that $B$ is
homogeneous with
\[
\dim_{H}B=\dim_{P}B=\lim_{k\rightarrow\infty}\frac{\log n_{1}\cdots n_{k}%
}{-\log c_{1}\cdots c_{k}}=\frac{\log5}{\log6}.
\]
Furthermore, assume that for every $i_{1}\cdots i_{k-1}\in\Omega_{k-1}$, the
subintervals
\[
J_{i_{1}\cdots i_{k-1}1},\cdots,J_{i_{1}\cdots i_{k-1}n_{k}}%
\]
are uniformly distributed in $J_{i_{1}\cdots i_{k-1}}$ from left to right.
Then $J_{i_{1}\cdots i_{k-1}(\frac{n_{k}+1}{2})}$ and $J_{i_{1}\cdots i_{k-1}%
}$ have the same middle point $y_{i_{1}\cdots i_{k-1}}$.

For the Moran measure $\nu$, we calculate that
\begin{equation}
\frac{\nu(B(y_{i_{1}\cdots i_{k_{m}}},(1/6)^{k_{m}}/2))}{\nu(B(y_{i_{1}\cdots
i_{k_{m}}},(1/6)^{t_{m}}/2))}=3^{m}. \label{xl}%
\end{equation}
Suppose $A$ is Ahlfors--David $s$-regular with $s\in(\log3/\log6,\log
5/\log6).$ Then for any $x\in A,$
\begin{equation}
\frac{\mu(B(x,(1/6)^{k_{m}}/2))}{\mu(B(x,(1/6)^{t_{m}}/2))}\geq\xi(6^{m})^{s}
\label{xl2}%
\end{equation}
for some constant $\xi.$

If (\ref{zip}) were true, by (\ref{xl}) and (\ref{xl2}) we would obtain that
$s\leq\log3/\log6.$ It is a contradiction.
\end{proof}

\subsection{Proof of embedding theorem}

\

Before the proof of Theorem \ref{T:A emb B}, we give a technical lemma as follows.

Suppose $B$ is homogeneous with the Borel measure $\nu.$ Let
\[
\overline{\nu}(r)=\sup_{x\in B}\nu(x,r)\text{ and }\underline{\nu}%
(r)=\inf_{x\in B}\nu(x,r).
\]

\begin{lemma}
Suppose that $A$ and $B$ are homogeneous sets. For any $\varepsilon>0$ and $\eta>0$ small enough, let $E(\eta)\in
\mathcal{M}(J,\{n_{k}\}_{k},\{c_{k}\}_{k})$ be the Moran set constructed in the proof of
Theorem $\ref{T:E Emb A}$, which is bilipschitz equivalent to $A(\eta)\subset A$. If
\[
P(B,\eta)\geq n_{1}\text{ and }\frac{\underline{\nu}(\eta^{k-1}/2)}%
{\overline{\nu}(2\eta^{k})}\geq n_{k}\text{ for all }k\geq2,
\]
then $E(\eta)\hookrightarrow B$, and thus $A(\eta)\hookrightarrow B$.
\end{lemma}


We turn to the proof of Theorem \ref{T:A emb B}.

\begin{proof}
[Proof of the first part of Theorem \ref{T:A emb B}]\

Without loss of generality, we may assume that $|A|=|B|=1.$ Let $\eta_{1}$ be
defined in (\ref{txt2}). Using the above lemma, by (\ref{E:nk geq 2}) it
suffices to show that if $\eta(\leq\eta_{1})$ is small enough, then%
\begin{equation}
n_{1}=P_{A}\leq P_{B}=P(B,\eta) \label{jsj}%
\end{equation}
and
\begin{equation}
\frac{\overline{\mu}(\eta^{k-1})}{\underline{\mu}(\eta^{k})}\leq
\frac{\underline{\nu}(\eta^{k-1}/2)}{\overline{\nu}(2\eta^{k})}\text{ for all
}k\geq\text{2}. \label{jsj2}%
\end{equation}

To obtain (\ref{jsj}), noticing that $P_{A}\asymp\frac{\mu(A)}{\mu(B(x^{\ast
},\eta))}$ and $P_{B}\asymp\frac{\nu(B)}{\nu(B(y^{\ast},\eta))}$, we only need
to check that
\[
\limsup_{\eta\rightarrow0}\left\vert \frac{\log\mu(B(x^{\ast},\eta))}{\log
\nu(B(y^{\ast},\eta))}\right\vert =\limsup_{\eta\rightarrow0} \frac{\alpha
_{A}(\eta)}{\alpha_{B}(\eta)} <1,
\]
which follows from (\ref{rar}) by fixing $r$ and letting $r^{\prime
}\rightarrow0$.

To obtain (\ref{jsj2}), note that $\frac{\overline{\mu}(\eta^{k-1}%
)}{\underline{\mu}(\eta^{k})}\asymp\frac{\mu(B(x^{\ast},\eta^{k-1}))}%
{\mu(B(x^{\ast},\eta^{k}))}$ and $\frac{\underline{\nu}(\eta^{k-1}%
/2)}{\overline{\nu}(2\eta^{k})}\asymp\frac{\nu(B(y^{\ast},\eta^{k-1}))}%
{\nu(B(y^{\ast},\eta^{k}))}$, we only need to find $\eta_{2}$ such that
\begin{equation}
\sup_{k\geq1,\eta<\eta_{2}}\left\vert \frac{\log\mu(B(x^{\ast},\eta
^{k-1}))-\log\mu(B(x^{\ast},\eta^{k}))}{\log\nu(B(y^{\ast},\eta^{k-1}%
))-\log\nu(B(y^{\ast},\eta^{k}))}\right\vert <1. \label{E:t5}%
\end{equation}
In fact, as at the end of the proof of Lemma \ref{L:blip}, we get (\ref{E:t5})
by using (\ref{rar}).
\end{proof}

In order to prove the second part, we need the following easily proved key property
\cite{Mattila Saarenen} by Mattila and Saaranen on the decomposition of a
uniformly disconnected set. The reader can refer to \cite{Wang Xi N} for a proof.

\begin{lemma}
\label{L:Uni Dis}\cite{Mattila Saarenen} Suppose $A$ is a uniformly disconnected compact subset of a metric space
with constants $C>1$ and $r^{\ast}>0$ in $(\ref{E:ud}).$ If $E$ is a subset of $A$ and $0<r<r^{\ast}$ a number satisfying $\mathrm{d}(E,A\backslash E)>Cr$, then there are sets
$\{E_{i}\}_{i=1}^{m}$ and balls $\{B(x_{i},r)\}_{i=1}^{m}$ satisfying

\begin{enumerate}
\item $E=\bigcup_{i=1}^{m}E_{i};$

\item $\mathrm{d}(E_{i},E_{j})>r$ for all $i\neq j;$

\item $x_{i}\in E_{i}$ and $E\cap B(x_{i},r)\subset E_{i}\subset B(x_{i},Cr)$ for
all $i.$
\end{enumerate}
\end{lemma}

Suppose $\{r_{k}\}_{k\geq1}$ is a sequence of positive numbers decreasing to
zero with $r_{1}<r^{\ast}$ and $r_{k}/r_{k+1}>C$ for all $k\geq1.$ We shall
give a decomposition of the uniformly disconnected set $A$ with respect to
$\{r_{k}\}_{k\geq1}.$

Set $\Lambda_{0}=\{\emptyset\}$ with empty word $\emptyset.$ Using Lemma
\ref{L:Uni Dis} with $E=A$ and $r=r_{1}$ we get sets $\{A_{i_{1}}\}_{i_{1}%
=1}^{m_{A}}$ and balls $\{B(x_{i_{1}},r_{1})\}_{i_{1}=1}^{m_{A}}$ satisfying

\begin{enumerate}
\item $A=\bigcup_{i_{1}=1}^{m_{A}}A_{i_{1}};$

\item $\mathrm{d}(A_{i_{1}},A_{j_{1}})>r_{1}(>Cr_{2})$ for all $i_{1}\neq
j_{1};$

\item $x_{i_{1}}\in A_{i_{1}}$ and $A\cap B(x_{i_{1}},r_{1})\subset A_{i_{1}}\subset
B(x_{i_{1}},Cr_{1})$ for all $i_{1}.$
\end{enumerate}

Set $m_{\emptyset}=m_{A}$ and $\Lambda_{1}=\{1,2,\cdots,m_{\emptyset}\}.$

For $k\geq2,$ assume that for $k-1$ we have got the sets $\{A_{i_{1}\cdots
i_{k-1}}\}_{i_{1}\cdots i_{k-1}\in\Lambda_{k-1}}$ and balls $\{B(x_{i_{1}%
\cdots i_{k-1}},r_{k-1})\}_{i_{1}\cdots i_{k-1}\in\Lambda_{k-1}}.$ By
induction, we can do the same work to every $A_{i_{1}\cdots i_{k-1}}$ with
$r=r_{k}.$ Let $\Lambda_{k}=\{i_{1}\cdots i_{k}:i_{j}\in\mathbb{N}\cap
\lbrack1,m_{i_{1}\cdots i_{j-1}}]$ for all $1\leq j\leq k\}.$

Since $r_{k}/r_{k+1}>C,$ using Lemma \ref{L:Uni Dis} again and again, we get
the decomposition of $A.$ There exist sets $A_{i_{1}\cdots i_{k}}$ and points
$x_{i_{1}\cdots i_{k}}$ such that for all $k\geq1,$
\begin{align*}
&  A=\bigcup_{i_{1}\cdots i_{k}\in\Lambda_{k}}A_{i_{1}\cdots i_{k}},\\
&  \mathrm{d}(A_{i_{1}\cdots i_{k-1}i_{k}},A_{i_{1}\cdots i_{k-1}j_{k}}%
)>r_{k}(>C r_{k+1})\text{ if }i_{k}\neq j_{k},\\
&  A_{i_{1}\cdots i_{k}i_{k+1}}\subset A_{i_{1}\cdots i_{k}},\\
&  x_{i_{1}\cdots i_{k}}\in A_{i_{1}\cdots i_{k}},\\
&  A\cap B(x_{i_{1}\cdots i_{k}},r_{k})\subset A_{i_{1}\cdots i_{k}}\subset
B(x_{i_{1}\cdots i_{k}},Cr_{k}).
\end{align*}
We denote $\Lambda=\bigcup_{k\geq0}\Lambda_{k}.$

\begin{proof}
[Proof of the second part of Theorem \ref{T:A emb B}]\

For any $\eta\in(0, \min(1/C,r^{\ast})),$ we get the decomposition of $A$ with
respect to $\{\eta^{k}\}_{k\geq1}.$ Note that $m_{A}\leq\frac{\mu
(A)}{\underline{\mu}(\eta)}$ and $m_{i_{1}\cdots i_{k}}\leq\frac{\overline
{\mu}(C\eta^{k})}{\underline{\mu}(\eta^{k+1})}$ for all $i_{1}\cdots i_{k}%
\in\Lambda\backslash\Lambda_{0}$. Hence, as in the first part of the proof,
there exists an $\eta_{3}\in(0,1/C),$ such that if $\eta$ satisfies $\eta
<\min\{\eta_{1},\eta_{3}\},$ then%
\[
m_{A}\leq P(B,\eta)\text{ and }\frac{\overline{\mu}(C\eta^{k-1})}%
{\underline{\mu}(\eta^{k})}\leq\frac{\underline{\nu}(\eta^{k-1}/2)}%
{\overline{\nu}(2\eta^{k})}\text{ for all }k\geq2.
\]

Corresponding to the decomposition of $A,$ we get a collection $\{B(y_{i_{1}%
\cdots i_{k}},\eta^{k})\}_{i_{1}\cdots i_{k}\in\Lambda\backslash\Lambda_{0}}$
of balls in $B$ as in Step 1 in the proof of Theorem \ref{T:E Emb A}, satisfying

\begin{enumerate}
\item For every $i_{1}\in\Lambda_{1},$ $y_{i_{1}}\in B$, and $B(y_{i_{1}}%
,\eta)\cap B(y_{j_{1}},\eta)=\varnothing$ for all $i_{1}\neq j_{1}$;

\item When $k\geq2,$ for every $i_{1}\cdots i_{k-1}\in\Lambda_{k-1},$
\end{enumerate}

\begin{itemize}
\item $y_{i_{1}\cdots i_{k-1}i_{k}}\in B(y_{i_{1}\cdots i_{k-1}},\eta
^{k-1}/2)\cap B$ for all $1\leq i_{k}\leq m_{i_{1}\cdots i_{k-1}}$;

\item $B(y_{i_{1}\cdots i_{k-1}i_{k}},\eta^{k})\cap B(y_{i_{1}\cdots
i_{k-1}j_{k}},\eta^{k})=\varnothing$ for all $i_{k}\neq j_{k}.$
\end{itemize}

We define
\[
B(\eta)=\bigcap\nolimits_{k\geq1}\bigcup\nolimits_{i_{1}\cdots i_{k}\in
\Lambda_{k}}B(y_{i_{1}\cdots i_{k}},\eta^{k}).
\]
Noting that
\[
A=\bigcap\nolimits_{k\geq1}\bigcup\nolimits_{i_{1}\cdots i_{k}\in\Lambda_{k}%
}A_{i_{1}\cdots i_{k}},
\]
we can check as in Step 2 in the proof of Theorem \ref{T:E Emb A} that the
natural bijection between $A$ and $B(\eta)$ is bilipschitz.
\end{proof}

\subsection{Uniform disconnectedness}

\

In the following proof, we use the idea of \cite{Mattila Saarenen} by Mattila
and Saaranen.

\begin{proof}
[Proof of Lemma \ref{L:<1 Uni dis}]$\ $

By (\ref{E:equ}), we may assume that there exists $r_{0}\in(0,1)$ such that
\begin{equation}
\sup_{r^{\prime}<r_{0}r^{\prime\prime}<r^{\prime\prime}<r_{0}}\left\vert
\frac{\log\overline{\mu}(r^{\prime\prime})-\log\underline{\mu}(r^{\prime}%
)}{\log r^{\prime\prime}-\log r^{\prime}}\right\vert \leq1-\gamma\text{ with
}\gamma>0. \label{E:jsj10}%
\end{equation}

Take an integer $l$ large enough such that
\begin{equation}
\frac{\log l-\log3}{\log(l+2)}>1-\gamma\text{ and }\frac{1}{l+2}<r_{0}.
\label{E:jsj11}%
\end{equation}

For any $r<r_{0}/(l+2)$ and $x\in A,$ let
\[
B_{0}=B(x,r),\text{ \ \ }B_{i}=B(x,(i+1)r)\backslash B(x,ir)\text{ }(1\leq
i\leq l+1).
\]
As in \cite{Mattila Saarenen}, we only need to verify

\begin{claim}
There must be an $i_{0}\in\{1,\cdots,l\}$ such that $A\cap B_{i_{0}%
}=\varnothing.$
\end{claim}

Otherwise, there exists $x_{i}\in$ $A\cap B_{i}$ whenever $1\leq i\leq l$.
Then
\begin{align*}
l\underline{\mu}(r)\leq\sum_{i=1}^{l}\mu(B(x_{i},r))  &  \leq\sum_{i=1}^{l}%
\mu(B_{i-1}\cup B_{i}\cup B_{i+1})\\
&  \leq3\mu(B(x,(l+2)r))\leq3\overline{\mu}((l+2)r).
\end{align*}
Therefore,
\[
1-\gamma<\frac{\log l-\log3}{\log(l+2)}\leq\left\vert \frac{\log\overline{\mu
}((l+2)r)-\log\underline{\mu}(r)}{\log(l+2)r-\log r}\right\vert \leq1-\gamma.
\]
This is a contradiction. Then the claim is proved, and thus (\ref{E:ud}) holds
with $C=l.$ That means $A$ is uniformly disconnected.
\end{proof}

We will construct a Moran set $E$ such that $\dim_{H}E=\dim_{P}E<1$ but $E$ is
not uniformly disconnected.

\begin{example}
\label{Ex:ud}

Let $k_{m}=m^{3}$ and $t_{m}=k_{m}+m$ for all $m.$ We take%
\[
(n_{k},c_{k})=\left\{
\begin{array}
[c]{ll}%
(3,1/3-1/(6m)) & \text{if }k\in\lbrack k_{m}+1,t_{m}]\text{ for some }m,\\
(3,1/6) & \text{otherwise.}%
\end{array}
\right.
\]
Let $E\in\mathcal{M}([0,1],\{n_{k}\}_{k},\{c_{k}\}_{k}).$ Then it follows from
Propositions $\ref{P:Mor is Homo}$ and $\ref{P:Prop of Homo}$ that $E$ is
homogeneous with
\[
\dim_{H}E=\dim_{P}E=\lim_{k\rightarrow\infty}\frac{\log n_{1}\cdots n_{k}%
}{-\log c_{1}\cdots c_{k}}=\frac{\log3}{\log6}<1.
\]
Assume that for every word $i_{1}\cdots i_{k-1}$, the subintervals
$J_{i_{1}\cdots i_{k-1}1},\cdots,J_{i_{1}\cdots i_{k-1}n_{k}}$ are uniformly
distributed in $J_{i_{1}\cdots i_{k-1}}$ from left to right. If we consider
the middle point $1/2$ and the largest gap in the interval $J_{i_{1}\cdots
i_{k_{m}}}$ with $1/2\in J_{i_{1}\cdots i_{k_{m}}},$ then, since
$1-3c_{(k_{m}+1)}\rightarrow0$ as $m\rightarrow\infty,$ we clearly see that we
can not find a uniform disconnectedness constant $C > 1.$
\end{example}

\section{Quasi-Lipschitz equivalence of homogeneous sets}

In this section, we will prove Theorem \ref{T:quasi-lip}. Without loss of
generality, we always assume that $A=X$ and $B=Y.$ We say that when
$r,r^{\prime}\rightarrow0,$
\[
g(r,r^{\prime})\rightarrow a\Leftrightarrow\bar{g}(r,r^{\prime})\rightarrow
b,
\]
if for any $\epsilon>0$ there exists an $\eta>0$ such that $\left\vert \bar
{g}(r,r^{\prime})-b\right\vert <\epsilon$ whenever $\max(\left\vert
g(r,r^{\prime})-a\right\vert ,|r|,|r^{\prime}|)<\eta$ and such that
$\left\vert g(r,r^{\prime})-a\right\vert <\epsilon$ whenever $\max(\left\vert
\bar{g}(r,r^{\prime})-b\right\vert ,|r|,|r^{\prime}|)<\eta$.

\begin{lemma}
\label{L:r/r=1}For any $x\in A,$ when $r,r^{\prime}\rightarrow0,$
\[
\frac{\log r^{\prime}}{\log r}\rightarrow1\Leftrightarrow\frac{\log
\mu(B(x,r^{\prime}))}{\log\mu(B(x,r))}\rightarrow1.
\]

\end{lemma}

\begin{proof}
Suppose that $r\in((\kappa_{A})^{k},(\kappa_{A})^{k-1}]$ and $r^{\prime}%
\in((\kappa_{A})^{k^{\prime}},(\kappa_{A})^{k^{\prime}-1}].$ Then%
\[
\frac{\log r}{k\log\kappa_{A}}\rightarrow1,\frac{\log r^{\prime}}{k^{\prime
}\log\kappa_{A}}\rightarrow1\text{ as }r,r^{\prime}\rightarrow0.
\]
On the other hand, $\mu(B(x,(\kappa_{A})^{k}))\leq\mu(B(x,r))\leq
\mu(B(x,(\kappa_{A})^{k-1}))$ and%
\[
\mu(B(x,(\kappa_{A})^{k}))\geq\Delta_{A}^{-1}\mu(B(x,(\kappa_{A})^{k-1}))
\]
due to (\ref{***}) in Definition \ref{D:Main}. Thus%
\[
\frac{\log\mu(B(x,r))}{\log\mu(B(x,(\kappa_{A})^{k}))}\rightarrow1,\frac
{\log\mu(B(x,r^{\prime}))}{\log\mu(B(x,(\kappa_{A})^{k^{\prime}}))}%
\rightarrow1\text{ as }r,r^{\prime}\rightarrow0.
\]

It suffices to verify that when $k,k^{\prime}\rightarrow\infty,$
\[
\frac{k^{\prime}}{k}\rightarrow1\Leftrightarrow\frac{\log\mu(B(x,(\kappa
_{A})^{k^{\prime}}))}{\log\mu(B(x,(\kappa_{A})^{k}))}\rightarrow1.
\]
For $k>k^{\prime},$ using (\ref{***}), we have $(\delta_{A})^{k-k^{\prime}%
}\leq\frac{\mu(B(x,(\kappa_{A})^{k^{\prime}}))}{\mu(B(x,(\kappa_{A})^{k}%
))}\leq(\Delta_{A})^{k-k^{\prime}},$ i.e.,
\begin{equation}
(k-k^{\prime})\log\delta_{A}\leq\log\frac{\mu(B(x,(\kappa_{A})^{k^{\prime}}%
))}{\mu(B(x,(\kappa_{A})^{k}))}\leq(k-k^{\prime})\log\Delta_{A}. \label{E:un}%
\end{equation}
Using (\ref{E:un}) and $0<\underline{\lim}_{k\rightarrow\infty}\alpha
_{A}(x,(\kappa_{A})^{k})\leq\overline{\lim}_{k\rightarrow\infty}\alpha
_{A}(x,(\kappa_{A})^{k})<\infty$, we have
\begin{align*}
\frac{\log\mu(B(x,(\kappa_{A})^{k^{\prime}}))}{\log\mu(B(x,(\kappa_{A})^{k}%
))}\rightarrow1  &  \Leftrightarrow\frac{\log\frac{\mu(B(x,(\kappa
_{A})^{k^{\prime}}))}{\mu(B(x,(\kappa_{A})^{k}))}}{\log\mu(B(x,(\kappa
_{A})^{k}))}\rightarrow0\\
&  \Leftrightarrow\left(  \frac{k-k^{\prime}}{k}\right)  \frac{\log(\kappa
_{A})^{k}}{\log\mu(B(x,(\kappa_{A})^{k}))}\rightarrow0\\
&  \Leftrightarrow\left(  \frac{k-k^{\prime}}{k}\right)  \frac{1}{\alpha
_{A}(x,(\kappa_{A})^{k})}\rightarrow0\\
&  \Leftrightarrow\frac{k^{\prime}}{k}\rightarrow1.
\end{align*}

\end{proof}

\subsection{Proof of equivalence theorem: necessity}

\

By the definition of quasi-Lipschitz equivalence, we can find a bijection
$f:A\rightarrow B$ and a non-decreasing function $\beta:\mathbb{R}%
^{+}\rightarrow\mathbb{R}^{+}$ with $\lim_{r\rightarrow0}\beta(r)=0$ such that
for every pair of distinct points $x_{1},x_{2}\in A,$%
\begin{equation}
1-\beta(\text{\textrm{d}}_{A}(x_{1},x_{2}))\leq\frac{\log\text{\textrm{d}}%
_{B}(f(x_{1}),f(x_{2}))}{\log\text{\textrm{d}}_{A}(x_{1},x_{2})}\leq
1+\beta(\text{\textrm{d}}_{A}(x_{1},x_{2}))\label{main}%
\end{equation}
and
\begin{equation}
1-\beta(\text{\textrm{d}}_{B}(f(x_{1}),f(x_{2})))\leq\frac{\log
\text{\textrm{d}}_{A}(x_{1},x_{2})}{\log\text{\textrm{d}}_{B}(f(x_{1}%
),f(x_{2}))}\leq1+\beta(\text{\textrm{d}}_{B}(f(x_{1}),f(x_{2}%
))).\label{mainmain}%
\end{equation}

For any $x\in A$ and $r>0$ small enough, we conclude that%
\begin{equation}
B(f(x),r)\subset f(B(x,r^{1-\beta(r)}))\text{ and }f(B(x,r))\subset
B(f(x),r^{1-\beta(r)}).\label{gggg}%
\end{equation}
In fact, we assume that $r$ and $\beta(r)$ are small enough. Firstly, we
verify that $B(f(x),r)\subset f(B(x,r^{1-\beta(r)})).$ For any $f(x^{\prime
})\in B(f(x),r)$ with $f(x^{\prime})\neq f(x)$, we have $0<\text{\textrm{d}%
}_{B}(f(x),f(x^{\prime}))\leq r$. By (\ref{mainmain}), we have
\[
\frac{\log\text{\textrm{d}}_{A}(x,x^{\prime})}{\log\text{\textrm{d}}%
_{B}(f(x),f(x^{\prime}))}\geq1-\beta(\text{\textrm{d}}_{B}(f(x),f(x^{\prime
})))\geq1-\beta(r),
\]
since the function $\beta$ is non-decreasing,
\[
\text{\textrm{d}}_{A}(x,x^{\prime})\leq\left(  \text{\textrm{d}}%
_{B}(f(x),f(x^{\prime}))\right)  ^{1-\beta(r)}\leq r^{1-\beta(r)}.
\]
Then $x^{\prime}\in B(x,r^{1-\beta(r)})$, and thus $B(f(x),r)\subset
f(B(x,r^{1-\beta(r)}))$. In the same way, using (\ref{main}), we have
$f(B(x,r))\subset B(f(x),r^{1-\beta(r)}).$

Using (\ref{gggg}), we have
\[
N(B,r^{1-\beta(r)})\leq N(A,r)\text{ and }N(A,r^{1-\beta(r)})\leq N(B,r).
\]
Since $\beta(r)\downarrow0$ when $r\downarrow0,$ using (2) of Proposition
\ref{P:Prop of Homo}, we have
\[
\chi(A,B)=0.
\]

\subsection{Proof of equivalence theorem: sufficiency}

$\ $

Without loss of generality we may assume that $|A|=|B|=1$. Let $\Sigma
=\{0,1\}^{\mathbb{N}}=\{w_{1}w_{2}\cdots:w_{i}\in\{0,1\}$ for all $i\geq1\}$
be a symbolic system equipped with the metric $D(x,y)=2^{-\min\{i\in
\mathbb{N}:w_{i}\neq\omega_{i}\}}$ for distinct points $x=w_{1}w_{2}\cdots,$
$y=\omega_{1}\omega_{2}\cdots$. Given two words $u=u_{1}\cdots u_{m}$ and
$v=v_{1}\cdots v_{n}$ and an infinite word $w=w_{1}w_{2}\cdots$, we write
\[
u\ast v=u_{1}\cdots u_{m}v_{1}\cdots v_{n} \quad\text{and}\quad u\ast
w=u_{1}\cdots u_{m}w_{1}w_{2}\cdots.
\]
The set $\{u\ast w: w\in\Sigma\}$ is called the cylinder determined by $u$,
and the length $m$ of $u$ is called the length of this cylinder.

Choose any $\eta\in(0,\min(1/C,r^{\ast})),$ where $C$ is the uniform
disconnectedness constant of $A$. Then we can get a decomposition of $A$ with
respect to $\{\eta^{k^{2}}\}_{k\geq1}$ (see the discussion after Lemma
\ref{L:Uni Dis}). Corresponding to the decomposition, we will give a
decomposition of $\Sigma$ and construct a quasi-Lipschitz bijection from
$\Sigma$ onto $A$. With the same work to $B,$ we can prove that the resulting
bijection between $A$ and $B$ is quasi-Lipschitz.

Now, for all $k\geq1\ $we have
\begin{equation}
\frac{\underline{\mu}(\eta^{(k-1)^{2}})}{\overline{\mu}(C\eta^{k^{2}})}\leq
m_{i_{1}\cdots i_{k-1}}\leq\frac{\overline{\mu}(C\eta^{(k-1)^{2}}%
)}{\underline{\mu}(\eta^{k^{2}})}. \label{E: number of m}%
\end{equation}
By choosing $\eta$ small enough we may assume that $m_{i_{1}\cdots i_{k-1}}>2$
for each $k\geq1$. Assume that $p_{i_{1}\cdots i_{k-1}}\geq1$ is the integer
satisfying
\begin{equation}
2^{p_{i_{1}\cdots i_{k-1}}}<m_{i_{1}\cdots i_{k-1}}\leq2^{1+p_{i_{1}\cdots
i_{k-1}}}. \label{E:m and p}%
\end{equation}

\textbf{Step 1.} According to the decomposition of $A,$ we give a
decomposition of $\Sigma$.

Set $\Sigma_{\emptyset}=\Sigma$ and $l_{\emptyset}=0.$ Denote all the words in
$\{0,1\}^{p_{\emptyset}}$ by $\pi_{1},\cdots,\pi_{2^{p_{\emptyset}}}.$ Then
the words
\[
\pi_{1}\ast0,\pi_{1}\ast1,\cdots,\pi_{m}\ast0,\pi_{m}\ast1,\pi_{m+1}%
,\cdots,\pi_{2^{p_{\emptyset}}}%
\]
give $m_{\emptyset}$ cylinders whose union is $\Sigma,$ where $m=m_{\emptyset
}-2^{p_{\emptyset}}.$

We denote these cylinders by $\{\Sigma_{i_{1}}\}_{i_{1}\in\Lambda_{1}}$ with
lengths $\{l_{i_{1}}\}_{i_{1}\in\Lambda_{1}}.$ It is clear that

\begin{enumerate}
\item $l_{i_{1}}=p_{\emptyset}$ or $1+p_{\emptyset}$ for all $1\leq i_{1}\leq
m_{\emptyset};$

\item $D(\Sigma_{i_{1}},\Sigma_{j_{1}})\geq2^{-(1+p_{\emptyset})}$ for all
$i_{1}\neq j_{1}.$
\end{enumerate}

For $k\geq2$, as usual, inductively assume that for $k-1$, we have got the
cylinders $\{\Sigma_{i_{1}\cdots i_{k-1}}\}_{i_{1}\cdots i_{k-1}\in
\Lambda_{k-1}}\text{ with lengths }\{l_{i_{1}\cdots i_{k-1}}\}_{i_{1}\cdots
i_{k-1}\in\Lambda_{k-1}}.$ With the same work to every $\Sigma_{i_{1}\cdots
i_{k-1}},$ we can find $m_{i_{1}\cdots i_{k-1}}$ cylinders $\Sigma
_{i_{1}\cdots i_{k-1}i_{k}}$ with lengths $l_{i_{1}\cdots i_{k-1}i_{k}}$ satisfying

\begin{enumerate}
\item $\Sigma_{i_{1}\cdots i_{k-1}}=\bigcup\nolimits_{i_{k}=1}^{m_{i_{1}\cdots
i_{k-1}}}\Sigma_{i_{1}\cdots i_{k-1}i_{k}};$

\item $D(\Sigma_{i_{1}\cdots i_{k-1}i_{k}},\Sigma_{i_{1}\cdots i_{k-1}j_{k}%
})\geq2^{-l_{i_{1}\cdots i_{k-1}}-(1+p_{i_{1}\cdots i_{k-1}})}$ if $i_{k}\neq
j_{k};$

\item $l_{i_{1}\cdots i_{k-1}i_{k}}-l_{i_{1}\cdots i_{k-1}}=p_{i_{1}\cdots
i_{k-1}}$ or $1+p_{i_{1}\cdots i_{k-1}}.$
\end{enumerate}

Then we get the decomposition of $\Sigma.$ There exist cylinders
$\Sigma_{i_{1}\cdots i_{k}}$ of lengths $l_{i_{1}\cdots i_{k}}$ such that for
all $k\geq1,$
\begin{align*}
&  \Sigma=\bigcup\nolimits_{i_{1}\cdots i_{k}\in\Lambda_{k}}\Sigma
_{i_{1}\cdots i_{k}},\\
&  D(\Sigma_{i_{1}\cdots i_{k-1}i_{k}},\Sigma_{i_{1}\cdots i_{k-1}j_{k}}%
)\geq2^{-l_{i_{1}\cdots i_{k-1}}-(1+p_{i_{1}\cdots i_{k-1}})}\text{ if }%
i_{k}\neq j_{k},\\
&  \Sigma_{i_{1}\cdots i_{k}i_{k+1}}\subset\Sigma_{i_{1}\cdots i_{k}},\\
&  l_{i_{1}\cdots i_{k-1}i_{k}}-l_{i_{1}\cdots i_{k-1}}=p_{i_{1}\cdots
i_{k-1}}\text{ or }1+p_{i_{1}\cdots i_{k-1}},\\
&  l_{i_{1}\cdots i_{k}}\geq k.
\end{align*}

\textbf{Step 2.} To verify the existence of the desired bijection between $A$
and $B$, we construct a bijection $f$\ from $\Sigma$ onto $A.$

Let $\Lambda^{\infty}$ be the collection of the infinite words $i_{1}\cdots
i_{k}\cdots$ with $i_{1}\cdots i_{k}\in\Lambda_{k}$ for all $k.$ For any
$i_{1}\cdots i_{k}\cdots\in\Lambda^{\infty},$ let $x_{i_{1}\cdots i_{k}\cdots
}\in A$ and $w_{i_{1}\cdots i_{k}\cdots}\in\Sigma$ be such that
\[
\{x_{i_{1}\cdots i_{k}\cdots}\}={\bigcap\nolimits_{k\geq1}}A_{i_{1}\cdots
i_{k}}\quad\text{and}\quad\{w_{i_{1}\cdots i_{k}\cdots}\}={\bigcap
\nolimits_{k\geq1}}\Sigma_{i_{1}\cdots i_{k}};
\]
note that $|\Sigma_{i_{1}\cdots i_{k}}|\leq2^{-k-1}\rightarrow0$ as
$k\rightarrow\infty$. In a natural way, we obtain a bijection $f$ from
$\Sigma$ onto $A,$ such that for any $i_{1}\cdots i_{k}\cdots\in
\Lambda^{\infty}$,
\[
f(w_{i_{1}\cdots i_{k}\cdots})=x_{i_{1}\cdots i_{k}\cdots}.
\]

In the next step, we will prove that for any distinct points $z_{1},z_{2}%
\in\Sigma,$
\begin{equation}
\frac{\alpha_{A}(\mathrm{d}_{A}(f(z_{1}),f(z_{2})))\log\mathrm{d}_{A}%
(f(z_{1}),f(z_{2}))}{\log D(z_{1},z_{2})}\rightarrow1\text{ uniformly as
}D(z_{1},z_{2})\rightarrow0. \label{E:alpha logd/ logD A}%
\end{equation}
Then we can also construct a bijection $g$ from $\Sigma$ onto $B$ in the same
way, with
\begin{equation}
\frac{\alpha_{B}(\mathrm{d}_{B}(g(z_{1}),g(z_{2})))\log\mathrm{d}_{B}%
(g(z_{1}),g(z_{2}))}{\log D(z_{1},z_{2})}\rightarrow1\text{ uniformly as
}D(z_{1},z_{2})\rightarrow0. \label{E:alpha logd/ logD B}%
\end{equation}
Now, we get a bijection $g\circ f^{-1}$ from $A$ onto $B.$ Using
(\ref{E:alpha logd/ logD A})--(\ref{E:alpha logd/ logD B}), Lemma
\ref{L:r/r=1} and the assumption $\chi(A,B)=0,$ we have
\begin{equation}
\label{E:dA/dB}\frac{\log\mathrm{d}_{B}(g(z_{1}),g(z_{2}))}{\log\mathrm{d}%
_{A}(f(z_{1} ),f(z_{2}))}\rightarrow1\text{ uniformly as }D(z_{1}%
,z_{2})\rightarrow0.
\end{equation}
Hence $A$ and $B$ are quasi-Lipschitz equivalent.

We will give the details of (\ref{E:dA/dB}) as follows:

According to the decompositions of $A$ and $\Sigma$, we know that the
bijection $f$ is continuous and thus is uniformly continuous. That is
\[
\mathrm{d}_{A}(f(z_{1}),f(z_{2}))\rightarrow0\text{ uniformly as }%
D(z_{1},z_{2})\rightarrow0.
\]
For the same reason,
\[
\mathrm{d}_{B}(g(z_{1}),g(z_{2}))\rightarrow0\text{ uniformly as }%
D(z_{1},z_{2})\rightarrow0.
\]

Firstly, by (6.4)--(6.5), we have
\[
\frac{\alpha_{A}(\mathrm{d}_{A}(f(z_{1}),f(z_{2})))\log\mathrm{d}_{A}%
(f(z_{1}),f(z_{2}))}{\alpha_{B}(\mathrm{d}_{B}(g(z_{1}),g(z_{2})))\log
\mathrm{d}_{B}(g(z_{1}),g(z_{2}))}\rightarrow1\text{ uniformly as }%
D(z_{1},z_{2})\rightarrow0.
\]
By the definition of $\chi$, we have $\chi(A,B)=0$ if and only if
$\lim_{r\rightarrow0}\frac{\alpha_{A}(r)}{\alpha_{B}(r)}=1$; then we get that
\[
\frac{\alpha_{A}(\mathrm{d}_{B}(g(z_{1}),g(z_{2})))}{\alpha_{B}(\mathrm{d}%
_{B}(g(z_{1}),g(z_{2})))}\rightarrow1\text{ uniformly as }D(z_{1}%
,z_{2})\rightarrow0.
\]
By the above two formulas, we obtain that
\[
\frac{\alpha_{A}(\mathrm{d}_{A}(f(z_{1}),f(z_{2})))\log\mathrm{d}_{A}%
(f(z_{1}),f(z_{2}))}{\alpha_{A}(\mathrm{d}_{B}(g(z_{1}),g(z_{2})))\log
\mathrm{d}_{B}(g(z_{1}),g(z_{2}))}\rightarrow1\text{ uniformly as }%
D(z_{1},z_{2})\rightarrow0,
\]
that is
\begin{equation}
\label{E:mudA/mudB}\frac{\log\mu(B(x,\mathrm{d}_{A}(f(z_{1}),f(z_{2}))))}%
{\log\mu(B(x,\mathrm{d}_{B}(g(z_{1}),g(z_{2}))))}\rightarrow1\text{ uniformly
as }D(z_{1},z_{2})\rightarrow0,
\end{equation}
for some fixed $x\in A$. Finally, by (\ref{E:mudA/mudB}) and Lemma
\ref{L:r/r=1}, we get (\ref{E:dA/dB}).

\textbf{Step 3.} We need to check (\ref{E:alpha logd/ logD A}).

For any given different points $z_{1},z_{2}\in\Sigma,$ suppose $i_{1}\cdots
i_{k-1}$ $(k\geq1)$ is the longest word such that $A_{i_{1}\cdots i_{k-1}}$
contains both $f(z_{1})$ and $f(z_{2}).$ Then $f(z_{1})\in A_{i_{1}\cdots
i_{k-1}i_{k}},$ $f(z_{2})\in A_{_{i_{1}\cdots i_{k-1}j_{k}}}$ with $i_{k}\neq
j_{k}.$ By (\ref{E: number of m}) and (\ref{E:m and p}),
\begin{align}
D(z_{1},z_{2})  &  \geq2^{-l_{i_{1}\cdots i_{k-1}}-(1+p_{i_{1}\cdots i_{k-1}%
})}\label{E:D >}\\
&  =2^{-(1+p_{i_{1}\cdots i_{k-1}})-(l_{i_{1}\cdots i_{k-1}}-l_{i_{1}\cdots
i_{k-2}})-\cdots-(l_{i_{1}}-l_{\emptyset})-l_{\emptyset}}\nonumber\\
&  \geq\prod_{i=1}^{k}\frac{\underline{\mu}(\eta^{i^{2}})}{2\overline{\mu
}(C\eta^{(i-1)^{2}})}.\nonumber
\end{align}
In the same way,
\begin{equation}
D(z_{1},z_{2})\leq2^{-l_{i_{1}\cdots i_{k-1}}}\leq\prod_{i=1}^{k-1}%
\frac{2\overline{\mu}(C\eta^{i^{2}})}{\underline{\mu}(\eta^{(i-1)^{2}})}.
\label{E:D<}%
\end{equation}
Then%
\begin{equation}
\mathrm{(I)}\leq\log D(z_{1},z_{2})\leq\mathrm{(II)}, \label{E:I logD II}%
\end{equation}
where
\begin{equation}
\mathrm{(I)}=\log\underline{\mu}(\eta^{k^{2}})-k\log2-\log\overline{\mu
}(C)+\sum_{i=1}^{k-1}\left(  \log\underline{\mu}(\eta^{i^{2}})-\log
\overline{\mu}(C\eta^{i^{2}})\right)  \label{E:I}%
\end{equation}
and
\begin{equation}
\mathrm{(II)=}\log\overline{\mu}(C\eta^{k^{2}})+(k-1)\log2+\sum_{i=1}%
^{k-2}\left(  \log\overline{\mu}(C\eta^{i^{2}})-\log\underline{\mu}%
(\eta^{i^{2}})\right)  . \label{E:II}%
\end{equation}
By Definition \ref{D:Main}, here $\log\underline{\mu}(\eta^{k^{2}})\geq
-ak^{2}+b$, $\log\overline{\mu}(C\eta^{k^{2}})\leq-a^{\prime}k^{2}+b^{\prime}$
and
\[
0\leq\log\overline{\mu}(C\eta^{i^{2}})-\log\underline{\mu}(\eta^{i^{2}})\leq c
\quad\text{ if }i\geq1
\]
with some constants $a,a^{\prime}>0$, $b,b^{\prime}\in\mathbb{R}$ and $c>0$.
Therefore, $k\rightarrow\infty,$ uniformly as $D(z_{1},z_{2})\rightarrow0.$

Notice that%
\[
\eta^{k^{2}}\leq\mathrm{d}_{A}(f(z_{1}),f(z_{2}))\leq|A_{i_{1}\cdots i_{k-1}%
}|\leq|B(x_{i_{1}\cdots i_{k-1}},C\eta^{(k-1)^{2}})|\leq2C\eta^{(k-1)^{2}};
\]
then
\begin{equation}
\frac{\log\mathrm{d}_{A}(f(z_{1}),f(z_{2}))}{k^{2}\log\eta}\rightarrow1\text{
uniformly as }D(z_{1},z_{2})\rightarrow0. \label{d rk}
\end{equation}

By (\ref{E:I logD II})--(\ref{E:II}) and the estimates related to (\ref{E:I})--(\ref{E:II}), we have
\begin{equation}\label{E:D/etak2}
  \frac{\log D(z_{1},z_{2})}{\alpha_{A}(\eta^{k^{2}})\cdot k^{2}\log\eta}\rightarrow 1 \text{ uniformly as }D(z_{1},z_{2})\rightarrow0.
\end{equation}
On the other hand, by (\ref{d rk}) and Lemma \ref{L:r/r=1} we have
\begin{equation}\label{E:df/etak2}
  \frac{\alpha_{A}(\mathrm{d}_{A}(f(z_{1}),f(z_{2})))}{\alpha_{A}(\eta^{k^{2}})}\rightarrow 1\text{ uniformly as }D(z_{1},z_{2})\rightarrow0.
\end{equation}
Now (\ref{d rk})--(\ref{E:df/etak2}) imply (\ref{E:alpha logd/ logD A}).

\end{document}